\documentclass[a4paper,11pt,DIV=13,headings=small]{scrartcl}

\usepackage{natbib}

\usepackage{amsmath}
\usepackage{amssymb}
\usepackage{amsthm}
\usepackage[T1]{fontenc}
\usepackage[utf8]{inputenc}
\usepackage{todonotes}
\usepackage{tabularx}
\newcolumntype{C}[1]{>{\centering\arraybackslash}p{#1}} 

\usepackage{tablefootnote}
\usepackage{footnote}

\makeatletter
\newcommand\footnoteref[1]{\protected@xdef\@thefnmark{\ref{#1}}\@footnotemark}
\makeatother

\usepackage{makecell}

\usepackage{lmodern}

\makeatletter
\def\cl@chapter{\@elt {theorem}}
\makeatother

\usepackage{xurl}
\usepackage{hyperref}
\usepackage{cleveref}
\usepackage{nameref}

\makeatletter
\Hy@AtBeginDocument{%
  \def\@pdfborder{0 -1 1}
}
\makeatother

\makeatletter
\g@addto@macro\bfseries{\boldmath}
\makeatother

\newtheorem{Theorem}{Theorem}
\newtheorem*{Theorem*}{Theorem}
\newtheorem{Assumption}{Assumption}
\newtheorem{Lemma}{Lemma}
\newtheorem{Corollary}{Corollary}

\newcommand{\R}{\mathbb{R}} 
\newcommand{\N}{\mathbb{N}} 
\newcommand{\E}{\mathbb{E}} 
\renewcommand{\P}{\mathbb{P}} 

\newcommand*\diff{\mathop{}\!\mathrm{d}}

\def\sectionautorefname#1\null{Section#1\null}
\def\equationautorefname#1#2\null{Eq.#1(#2\null)}

\DeclareMathOperator{\Var}{Var}
\DeclareMathOperator{\Bias}{Bias}
\DeclareMathOperator{\MSE}{MSE}
\DeclareMathOperator{\MISE}{MISE}

\DeclareMathOperator{\Poi}{Poi}
\DeclareMathOperator{\Beta}{Beta}
\DeclareMathOperator\supp{supp}
\def\Eq#1{Eq. (\ref{#1})}

\newcommand{\FourFigureVeranschaulichungHermite}[3][0.45\textwidth]{%
\begin{figure*}
	\includegraphics[width=#1]{#2_1.pdf}
	\hspace*{\fill}
	\includegraphics[width=#1]{#2_3.pdf}
	\hspace*{\fill} \\
	\includegraphics[width=#1]{#2_4.pdf}
	\hspace*{\fill}
	\includegraphics[width=#1]{#2_6.pdf}
	\caption{#3}
	\label{#2}
\end{figure*}
}%

\newcommand{\DoubleFigureHere}[3][0.45\textwidth]{%
	\begin{figure*}
		\includegraphics[width=#1]{#2_1.pdf}
		\includegraphics[width=#1]{#2_2.pdf}
		\caption{#3}
		\label{#2}
	\end{figure*}
}%

\newcommand{\SingleFigureHere}[4]{%
	\begin{figure*}
		\begin{center}
     	\includegraphics[width=#1]{#2.pdf}
		\vspace*{-5mm}
		\end{center}
		\caption{#3}
		\label{#4}
	\end{figure*}
}%

\begin{document}

\title{Smooth Distribution Function Estimation for Lifetime Distributions using Szasz-Mirakyan Operators}


\author{Ariane Hanebeck \\
        \small{Institute of Applied Mathematical Statistics}  \\
        \small{Technical University of Munich} \\
        \small{Ariane.Hanebeck@tum.de}
        \and
        Bernhard Klar \\
        \small{Institute of Stochastics}  \\
        \small{Karlsruhe Institute of Technology} \\
        \small{Bernhard.Klar@kit.edu}
        }

\maketitle

\begin{abstract}
In this paper, we introduce a new smooth estimator for continuous distribution functions on the positive real half-line using Szasz-Mirakyan operators, similar to Bernstein's approximation theorem. We show that the proposed estimator outperforms the empirical distribution function in terms of asymptotic (integrated) mean-squared error, and generally compares favourably with other competitors in theoretical comparisons. Also, we conduct the simulations to demonstrate the finite sample performance of the proposed estimator.
\end{abstract}
\noindent {\small {\itshape Keywords.}}\quad Distribution function estimation, Nonparametric, Szasz-Mirakyan operator, Hermite estimator, Mean squared error, Asymptotic properties

\section{Introduction}



This paper considers the nonparametric smooth estimation of continuous distribution functions on the positive real half line. Arguably, such distributions are the most important univariate probability models, occuring in  diverse fields such as life sciences, engineering, actuarial sciences or finance,  under various names such as life, lifetime, loss or survival distributions. The well-known compendium of \cite{johnsonContinuousUnivariateDistributions1994} treats in its first volume solely distributions on the positive half line with the exception of the normal and the Cauchy distribution. In the two volumes \cite{johnsonContinuousUnivariateDistributions1994,johnsonContinuousUnivariateDistributions1995} as well as in the compendiums about life and loss distributions of \cite{marshallLifeDistributionsStructure2007} and \cite{hoggLossDistributions1984}, respectively, an abundance of parametric models for the distribution of non-negative random variables and pertaining estimation methods can be found.

However, there is a paucity of nonparametric estimation methods especially tailored to this situation. It is the aim of this paper to close this gap by introducing a new nonparametric estimator for distribution functions on $[0,\infty)$ using Szasz-Mirakyan operators.

Let $X_1, X_2, ...$ be a sequence of independent and identically distributed (i.i.d.) random variables having an underlying unknown distribution function $F$  and associated density function $f$.
In the case of \textit{parametric} distribution function estimation, the model structure is already defined before knowing the data. It is for example known that the distribution will be of the form $\mathcal{N}(\mu,\sigma^2)$. The only goal is to estimate the parameters, here $\mu$ and $\sigma^2$.
Compared to this, in the \textit{nonparametric} setting, the model structure is not specified a priori but is determined only by the sample. In this paper, all the considered estimators are of nonparametric type.

The goal is to investigate properties of a random sample and its underlying distribution.
Of utmost importance is the probability $\P(a \leq X_1 \leq b)=F(b)-F(a)$, which can directly be estimated without the need to integrate as in the density estimation setting.
By taking the inverse of $F$, it is also possible to calculate quantiles
\begin{equation*}
    x_p=\inf\{x \in \R: F(x) \geq p\}=F^{-1}(p).
\end{equation*}
An important application of the inverse of $F$ is the so-called Inverse Transform Sampling.
It can be used to generate more samples than already given using the implication
\begin{equation*}
    Y \sim U[0,1] \Rightarrow F^{-1}(Y) \sim X_1.
\end{equation*}


The best-known distribution function estimators with well-established properties  are the empirical distribution function (EDF) and the kernel estimator.

The EDF is the simplest way to estimate the underlying distribution function, given a finite random sample $X_1,...,X_n, n \in \N$.
It is defined by
    \begin{equation*}
        F_n(x)=\frac{1}{n}\sum_{i=1}^n \mathbb{I}(X_i \leq x),
    \end{equation*}
where $\mathbb{I}$ is the indicator function.
This estimator is obviously not continuous. The kernel distribution function  estimator, however, is a continuous estimator.
The univariate kernel density estimator is defined by
\begin{equation*}
    f_{h,n}(x)=\frac{1}{nh}\sum_{i=1}^nK\Bigg(\frac{x-X_i}{h}\Bigg), x \in \R,
\end{equation*}
where the parameter $h >0$ is called the bandwidth and $K: \R \to \R$ is a kernel that has to fulfill specific properties (see, e.g., \cite{gramackiNonparametricKernelDensity2018}). It was first introduced by \cite{rosenblattRemarksNonparametricEstimates1956} and \cite{parzenEstimationProbabilityDensity1962}.

The idea is that the number of kernels is higher in regions with many samples, which leads to a higher density. The width and height of each kernel is determined by the bandwidth $h$. In the above case, the bandwidth is the same for all kernels.

To estimate the distribution function, the kernel density estimator is integrated. 
Hence, the kernel distribution function estimator is of the form
    \begin{equation*}
        \label{Eq:KernelEstimator}
        F_{h,n}(x)=\int_{-\infty}^x f_{h,n}(u) \diff u=\int_{-\infty}^x\frac{1}{nh}\sum_{i=1}^nK\Bigg(\frac{u-X_i}{h}\Bigg)\diff u=\frac{1}{n}\sum_{i=1}^n\mathbb{K}\Bigg(\frac{x-X_i}{h}\Bigg),
    \end{equation*}
where
      $\mathbb{K}(t)=\int_{-\infty}^tK(u)\diff u$
is a cumulative kernel function.
This estimator was first introduced by \cite{yamatoUniformConvergenceEstimator1973}. Different methods to choose the bandwidth in the case of the distribution function are given in \cite{altmanBandwidthSelectionKernel1995}, \cite{bowmanBandwidthSelectionSmoothing1998}, \cite{polanskyMultistagePlugBandwidth2000}, and \cite{tenreiroAsymptoticBehaviourMultistage2006}.

The two previous estimators can estimate distribution functions on any arbitrary real interval. The Bernstein estimator, on the other hand, is designed for functions on $[0,1]$.

The goal of the Bernstein estimator is the estimation of a distribution function $F$ with density $f$ supported on $[0,1]$, given a finite random sample $X_1,...,X_n, n \in \N$. It makes use of the following theorem.
\begin{Theorem*}
    \label{Theorem:Feller}
    If $u$ is a continuous function on $[0,1]$, then as $m \rightarrow \infty$,
    \begin{equation*}
    B_m(u;x)=\sum_{k=0}^m u\left(\frac{k}{m}\right)P_{k,m}(x) \rightarrow u(x)
    \end{equation*}
    uniformly for $x \in [0,1]$, where $P_{k,m}=\binom{m}{k}x^k(1-x)^{m-k}$ are the Bernstein basis polynomials.
\end{Theorem*}
Using this theorem, $F$ can be represented by the expression
    \begin{equation*}
        B_m(F;x)=\sum_{k=0}^m F\left(\frac{k}{m}\right)P_{k,m}(x),
    \end{equation*}
which converges to $F$ uniformly for $x \in [0,1]$. As the distribution function $F$ is unknown, the idea now is to replace $F$ with the EDF $F_n$. Following \cite{leblancEstimatingDistributionFunctions2012}, this leads to the Bernstein estimator
    \begin{equation*}
        \hat{F}_{m,n}(x)= \sum_{k=0}^m F_n\left(\frac{k}{m}\right)P_{k,m}(x).
    \end{equation*}

A further estimator is the Hermite estimator on the real half line that can be defined for different intervals.
It makes use of the so-called Hermite polynomials $H_k$ that are defined by
    \begin{equation*}
        H_k(x)=(-1)^ke^{x^2}\frac{d^k}{\diff x^k}e^{-x^2}.
    \end{equation*}
These polynomials are orthogonal under $e^{-x^2}$.
The normalized Hermite functions are given by
    \begin{equation*}
        \label{Eq:NormalizedHermite}
        h_k(x)=(2^kk!\sqrt{\pi})^{-1/2}e^{\frac{-x^2}{2}}H_k(x).
    \end{equation*}
They form an orthonormal basis for $L_2$. We define
    \begin{equation*}
        Z(x)=\frac{1}{\sqrt{2 \pi}}e^{\frac{-x^2}{2}} \, ,
        \alpha_k=\frac{\sqrt{\pi}}{2^{k-1}k!},
    \end{equation*}
    and
    \begin{equation*}
        a_k=\int_{-\infty}^{\infty}f(x)h_k(x)\diff x.
    \end{equation*}
Now, for $f \in L_2$,
    \begin{align}
        f(x)=\sum_{k=0}^{\infty}a_k h_k(x)=\sum_{k=0}^{\infty} \sqrt{\alpha_k} \cdot a_k H_k(x)Z(x)\label{Eq:InfiniteSum}.
    \end{align}
The infinite sum in \autoref{Eq:InfiniteSum} is not desirable. A truncation of the sum leads to the $N+1$ truncated expansion
    \begin{align*}
        f_N(x)&=\sum_{k=0}^Na_k h_k(x)=\sum_{k=0}^N \sqrt{\alpha_k}\cdot a_k H_k(x)Z(x).
    \end{align*}
The coefficients $a_k$ are chosen so that the $L_2$-distance between $f$ and $f_N$ is minimized. A detailed explanation can be found in Section 2.3 of \cite{davisFourierSeriesOrthogonal1963}. 
Now, the density estimator is of the form
    \begin{align*}
        \label{Eq:DensityGauss}
        \hat{f}_{N,n}(x)&=\sum_{k=0}^N \hat{a}_kh_k(x)=\sum_{k=0}^N \sqrt{\alpha_k}\cdot \hat{a}_kH_k(x)Z(x)
    \end{align*}
    with $\hat{a}_k=\frac{1}{n}\sum_{i=1}^n h_k(X_i)$
Using this, the Hermite distribution function estimators on the half line and the real line are defined by
    \begin{equation*}
        \hat{F}_{N,n}^H(x)=\int_{0}^x\hat{f}_{N,n}(t)\diff t, \enspace \text{and} \enspace \hat{F}_{N,n}^F(x)=\int_{-\infty}^x\hat{f}_{N,n}(t)\diff t
    \end{equation*}
respectively, following \cite{stephanouSequentialQuantilesHermite2017} \cite{stephanouPropertiesHermiteSeries2020a}.


More information on the different estimators can be found in the cited literature and  in \cite{hanebeckNonparametricDistributionFunction2020}. In the comparison in \autoref{Section:Comparison}, many properties of the estimators are listed.

In the case where a random variable $Y$ is supported on the compact interval $[a,b], a<b$, it can easily be restricted to $[0,1]$ by transforming $Y$ to $(Y-a)/(b-a)$. The back-transformation can be done without worrying about optimality or convergence rates.

However, in most cases, it is not enough to consider distributions on $[0,1]$. If the support of a random variable $Z$ is $(-\infty, \infty)$ or $[0,\infty)$, possible transformations to $(0,1)$ are $1/2+(1/\pi)\tan^{-1}Z$ and $Z/(1+Z)$, respectively.
Although the resulting random variable is supported on $(0,1)$, it is not clear what happens to optimality conditions and convergence rates after the back-transformation.

Another argument against nonlinear transformations is the loss of interpretability. Consider two random variables $Z_1$ and $Z_2$ on $[0,\infty)$, and the transformed quantities $Y_1=Z_1/(1+Z_1)$ and $Y_2=Z_2/(1+Z_2)$. If $Y_1$ is smaller than $Y_2$ in the (usual) stochastical order, it is not directly apparent if this also holds for $Z_1$ and $Z_2$.
Hence, such transformations have to be treated with care.

In this paper, we consider the Szasz estimator, as an alternative estimator of the distribution function on $[0,\infty)$.
The kernel estimator can also estimate functions on $[0,\infty)$ but is not specifically designed for this interval. To get satisfactory results, special boundary corrections in the point zero are necessary (see \cite{zhangEstimatingDistributionFunction2020}), which is not the case for the Szasz estimator.
The Hermite estimator on the real half line is designed for $[0, \infty)$, but theoretical results and simulations later show that the Szasz estimator performs better on the positive real line.

The paper is organized as follows.
In \autoref{Section:IdeaSzasz}, the approach and most important properties of the proposed estimator are explained. Then, in \autoref{Section:Derivations}, we derive asymptotic properties of the estimator. In \autoref{Section:Comparison}, the properties are compared with other estimators in a theoretical comparison, and then in a simulation study in \autoref{Section:Simulation}.
\autoref{Section:Conclusions} concludes the paper. Most proofs are postponed to the Appendix.

Throughout the paper, the notation $f=o(g)$ means that $\lim|f/g|=0$ as $m,n \rightarrow \infty$. A subscript (for example $f=o_x(g)$) indicates which parameters the convergence rate can depend on. Furthermore, the notation $f=O(g)$ means that $\limsup|f/g|<C$ for $m,n \rightarrow \infty$ and some $C \in (0,\infty)$. A subscript in this case means that $C$ could depend on the corresponding parameter.

\section{The Szasz Distribution Function Estimator} 
\label{Section:IdeaSzasz}
The idea of the estimator presented in this paper is similar to the Bernstein approach. The main difference is that instead of the Bernstein basis polynomials, we use Poisson probabilities. Hence, in the former case, we consider $\supp(f)=[0,1]$, while the latter case assumes $\supp(f)=[0,\infty)$. We make use of the following theorem that can be found in \cite{szaszGeneralizationBernsteinPolynomials1950}.
\begin{Theorem}
    \label{Theorem:Szasz}
    If $u$ is a continuous function on $(0,\infty)$ with a finite limit at infinity, then, as $m \rightarrow \infty$,
    \begin{equation*}
        S_m(u;x)=\sum_{k=0}^{\infty}u\left(\frac{k}{m}\right)V_{k,m}(x) \rightarrow u(x)
    \end{equation*}
    uniformly for $x \in (0,\infty)$, where $V_{k,m}(x)=e^{-mx}\frac{(mx)^k}{k!}$ for $k,m\in\N$.
\end{Theorem}
The operator $S_m(u;x)$ is called the Szasz-Mirakyan operator of the function $u$ at the point $x$.
One can expand \autoref{Theorem:Szasz} to a function $u$ being continuous on $[0,\infty)$ with $u(0)=0$. Then, $S_m(u;0)=0$ and with the continuity it holds that $S_m(u;x) \rightarrow u(x)$ uniformly for $x \in [0,\infty)$.
In particular, a continuous distribution function $F$ on $[0,\infty)$ can be represented by
\begin{equation}
\label{Eq:SmF}
S_m(F;x)=\sum_{k=0}^{\infty}F\left(\frac{k}{m}\right)V_{k,m}(x),
\end{equation}
which converges to $F$ uniformly for $x \in [0,\infty)$.
Then, a possible estimator of 
$F$ on $[0,\infty)$ is
\begin{equation*}
\hat{F}_{m,n}^S(x)
=\sum_{k=0}^{\infty}F_n\left(\frac{k}{m}\right)V_{k,m}(x),
\end{equation*}
replacing the unknown distribution function $F$ in the Szasz-Mirakyan operator \autoref{Eq:SmF} by the EDF $F_n$.
We call $\hat{F}_{m,n}^S$ the Szasz estimator. 
The sum is infinite but can be written as a finite sum as shown in the next subsection.

In the remainder of this paper, we make the following general assumption:

\begin{Assumption}
  \label{Assumption:3}
  The distribution function $F$ is continuous. The first and second derivatives $f$ and $f'$ of $F$ are continuous and bounded on $[0,\infty)$.
\end{Assumption}
Note that if only the convergence itself is important and we are not interested in deriving the convergence rate, it is enough to assume these properties on $(0,\infty)$.

\subsection{Basic Properties of the Szasz Estimator}

The behavior of the Szasz estimator $\hat{F}_{m,n}^S(x)$ at $x=0$ and for $x \rightarrow \infty$ is appropriate, since we get
\begin{align}
  \label{Eq:BoundaryPoisson}
  \hat{F}_{m,n}^S(0)&=0=F(0)=S_m(F;0), \nonumber\\
  \lim_{x \rightarrow \infty} \hat{F}_{m,n}^S(x)&=1=\lim_{x \rightarrow \infty} F(x)=\lim_{x \rightarrow \infty} S_m(F;x)
\end{align}
with probability one for all $m$. This means that bias and variance at the point $x=0$ are zero.

In the sequel, we use the gamma function
$\Gamma(z)=\int_0^{\infty}x^{z-1}e^{-x}dx$,
as well as the upper and lower incomplete gamma functions, defined by
\begin{align*}
  \Gamma(z,s)=\int_s^{\infty}x^{z-1}e^{-x}dx, \enspace \text{and} \enspace
  \gamma(z,s)=\int_0^{s}x^{z-1}e^{-x}dx,
\end{align*}
respectively.
The limit on the left side of \autoref{Eq:BoundaryPoisson} is one since
\begin{align*}
  \hat{F}_{m,n}^S(x)&=\sum_{k=0}^{\infty}F_n\left(\frac{k}{m}\right)V_{k,m}(x)=\frac{1}{n}\sum_{i=1}^n\sum_{k=0}^{\infty}\mathbb{I}\{k\geq mX_i\}V_{k,m}(x)\nonumber\\
  &=\frac{1}{n}\sum_{i=1}^n\sum_{k=\lceil mX_i\rceil}^{\infty}V_{k,m}(x)=\frac{1}{n}\sum_{i=1}^n\P(Y \geq \lceil mX_i\rceil)\\
  &=\frac{1}{n}\sum_{i=1}^n\frac{\gamma(\lceil mX_i\rceil,mx)}{\Gamma(\lceil mX_i\rceil)}\nonumber\xrightarrow{x \rightarrow \infty} 1,
\end{align*}
where the random variable $Y$ has a Poisson distribution with expected value $mx$  ($Y \sim \Poi(mx)$ for short). Since the above representation only contains a finite number of summands, it can be used to easily simulate the estimator.

The expectation of the Szasz operator is of course given by the expression $\E[\hat{F}_{m,n}^S(x)]=S_m(F;x)$ for $x\in[0,\infty)$.

It is worth noting that $\hat{F}_{m,n}^S(x)$ yields a proper continuous distribution function with probability one and for all values of $m$. The continuity of $\hat{F}_{m,n}^S(x)$ is obvious. Moreover, it follows from \autoref{Eq:BoundaryPoisson} and the next theorem that $0\leq\hat{F}_{m,n}^S(x)\leq~1$ for $x\in [0,\infty)$.
\begin{Theorem}
  The function $\hat{F}_{m,n}^S(x)$ is increasing in $x$ on $[0,\infty)$.
\end{Theorem}
\begin{proof}
  This proof is similar to the proof for the Bernstein estimator that can be found in \cite{babuApplicationBernsteinPolynomials2002}. Let
  \begin{equation*}
    g_n(0)=0, \enspace g_n\left(\frac{k}{m}\right)=F_n\left(\frac{k}{m}\right)-F_n\left(\frac{k-1}{m}\right), k=1,2,...,
  \end{equation*}
  and
  \begin{equation*}
    U_k(m,x)=\sum_{j=k}^{\infty}V_{j,m}(x)
    =\frac{1}{\Gamma(k)}\int_0^{mx}t^{k-1}e^{-t}\diff t.
  \end{equation*}
  The last equation holds because
  \begin{equation*}
    U_k(m,x)=1-\sum_{j=0}^{k-1}V_{j,m}(x)=1-\frac{\Gamma(k,mx)}{\Gamma(k)}=\frac{\gamma(k,mx)}{\Gamma(k)}.
  \end{equation*}
  It follows that $\hat{F}_{m,n}^S$ can be written as
  \begin{equation*}
    \hat{F}_{m,n}^S(x)=\sum_{k=0}^{\infty}g_n\left(\frac{k}{m}\right)U_k(m,x)
  \end{equation*}
  because
  \begin{align*}
    \sum_{k=0}^{\infty}g_n\left(\frac{k}{m}\right)U_k(m,x)&=\sum_{k=1}^{\infty}\left[F_n\left(\frac{k}{m}\right)-F_n\left(\frac{k-1}{m}\right)\right]\sum_{j=k}^{\infty}V_{j,m}(x)\nonumber\\
    &=\sum_{k=1}^{\infty}\sum_{j=k}^{\infty}F_n\left(\frac{k}{m}\right)V_{j,m}(x)-\sum_{k=0}^{\infty}\sum_{j=k}^{\infty}F_n\left(\frac{k}{m}\right)V_{j,m}(x)\nonumber\\
    &\enspace +\sum_{k=0}^{\infty}F_n\left(\frac{k}{m}\right)V_{k,m}(x)=\hat{F}_{m,n}^S(x).
  \end{align*}
  The claim follows as $g_n\left(\frac{k}{m}\right)$ is non-negative for at least one $k$ and $U_k(m,x)$ is increasing.
\end{proof}
The next theorem shows that $\hat{F}_{m,n}^S(x)$ is uniformly strongly consistent.
\begin{Theorem}
  If $F$ is a continuous probability distribution function on $[0,\infty)$, then
  \begin{equation*}
    \left\|\hat{F}_{m,n}^S-F\right\|\rightarrow 0 \enspace \text{a.s.}
  \end{equation*}
  for $m,n \rightarrow \infty$. We use the notation $\|G\|=\displaystyle\sup_{x\in[0,\infty)}|G(x)|$ for a bounded function $G$ on $[0,\infty)$.
\end{Theorem}
\begin{proof}
  The proof follows the proof of Theorem 2.1 in  \cite{babuApplicationBernsteinPolynomials2002}.
  It holds that
  \begin{equation*}
    \left\|\hat{F}_{m,n}^S-F\right\| \leq \left\|\hat{F}_{m,n}^S-S_m\right\|+\|S_m-F\|
  \end{equation*}
  and
  \begin{align*}
    \left\|\hat{F}_{m,n}^S-S_m\right\|&=\|\sum_{k=0}^{\infty} \left[F_n(k/m)-F(k/m)\right]V_{k,m}\|\nonumber\\
    &\leq \|F_n-F\|\cdot\|\sum_{k=0}^{\infty} V_{k,m}\|=\|F_n-F\|.
  \end{align*}
Since $\|F_n-F\|\rightarrow 0$ a.s. for $n \rightarrow \infty$ by the Glivenko-Cantelli theorem, the claim follows with \autoref{Theorem:Szasz}.
\end{proof}

\section{Asymptotic Properties of the Szasz estimator}
\label{Section:Derivations}

\subsection{Bias and Variance}

We now calculate the bias and the variance of the Szasz estimator $\hat{F}_{m,n}^S$ on the inner interval $(0,\infty)$, as we already know that bias and variance are zero for $x=0$. In the following lemma, we first find a different expression of $S_m$ that is similar to a result in \cite{lorentzBernsteinPolynomials1986}.

\begin{Lemma}
  \label{Lemma:LorentzS}
    We have, for $x \in (0,\infty)$ that
    \begin{align*}
      S_m(F;x)=F(x)+m^{-1}b^S(x)+o_x(m^{-1}),
    \end{align*}
    where $b^S(x)=\frac{xf'(x)}{2}$.
 \end{Lemma}

\begin{proof}
  Following the proof in \citet[Section 1.6.1]{lorentzBernsteinPolynomials1986}, Taylor's theorem gives
  \begin{align*}
    S_m(F;x)
    &=\sum_{k=0}^{\infty} F\left(\frac{k}{m}\right)V_{k,m}(x)=F(x)+ \sum_{k=0}^{\infty} \left(\frac{k}{m}-x\right) f(x) V_{k,m}(x) \\
    &\; + \frac{1}{2} f'(x)\sum_{k=0}^{\infty} \left(\frac{k}{m}-x\right)^2 V_{k,m}(x) + \sum_{k=0}^{\infty} o\left(\left(\frac{k}{m}-x\right)^2\right) V_{k,m}(x).
  \end{align*}
  The second summand, say, $S_2$, simplifies to $S_2=xf(x)-xf(x)=0$, because for $x\in[0,\infty)$ it holds that
  \begin{equation*}
    \sum_{k=0}^{\infty} \frac{k}{m}V_{k,m}(x)=\frac{1}{m}\E[Y]=x,
  \end{equation*}
  where $Y \sim \Poi(mx)$.
  The third term can be written as
  \begin{equation}
    \label{Eq:QuadratS}
    \sum_{k=0}^{\infty} \left(\frac{k}{m}-x\right)^2 V_{k,m}(x)=\frac{1}{m^2}\Var[Y]=\frac{x}{m}.
  \end{equation}
  For the last summand we know that
  \begin{align*}
    & \sum_{k=0}^{\infty} o\left(\left(\frac{k}{m}-x\right)^2\right) V_{k,m}(x) = o\left(\sum_{k=0}^{\infty} \left(\frac{k}{m}-x\right)^2 V_{k,m}(x)\right)  = o\left(\frac{x}{m}\right)=o_x(m^{-1})
  \end{align*}
  with \autoref{Eq:QuadratS}.
\end{proof}

The following theorem establishes asymptotic expressions for the bias and the variance of the Szasz estimator $\hat{F}_{m,n}^S$ as $m,n \rightarrow \infty$ are established. The statement is similar to Theorem 1 in \cite{leblancEstimatingDistributionFunctions2012}.
\begin{Theorem}
  \label{Theorem:BiasVarS}
  For each $x\in(0,\infty)$, the bias has the representation
  \begin{align*}
    \Bias\left[\hat{F}_{m,n}^S(x)\right]&=\E\left[\hat{F}_{m,n}\right] - F(x)=m^{-1}\frac{xf'(x)}{2}+o_x(m^{-1})\nonumber\\
    &=m^{-1}b^S(x)+o_x(m^{-1}).
  \end{align*}
  For the variance it holds that
  \begin{align*}
    \Var\left[\hat{F}_{m,n}^S(x)\right]=n^{-1}\sigma^2(x)-m^{-1/2}n^{-1}V^S(x)+o_x(m^{-1/2}n^{-1}),
  \end{align*}
  where
  \begin{equation*}
    \sigma^2(x)=F(x)(1-F(x)) \text{,} \enspace V^S(x)=f(x)\left[\frac{x}{\pi}\right]^{1/2}
  \end{equation*}
  and $b^S(x)$ is defined in \autoref{Lemma:LorentzS}.
\end{Theorem}
\noindent For the proof, see Section \nameref{Section:ProofsSzasz}.


\subsection{Asymptotic Normality}

Here, we turn our attention to the asymptotic behavior of the Szasz estimator. The next theorem is similar to Theorem 2 in \cite{leblancEstimatingDistributionFunctions2012} and shows the asymptotic normality of this estimator.
\begin{Theorem}
  \label{Theorem:AsySzasz}
  Let $x \in (0,\infty)$, such that $0 < F(x) < 1$. Then, for $m,n \rightarrow \infty$ it holds that
  \begin{equation*}
    n^{1/2}\left(\hat{F}_{m,n}^S(x)-\mathbb{E}[\hat{F}_{m,n}^S(x)]\right) = n^{1/2}\left(\hat{F}_{m,n}^S(x)-S_m(F;x)\right) \xrightarrow{D} \mathcal{N}\left(0,\sigma^2(x)\right),
  \end{equation*}
  where $\sigma^2(x)=F(x)(1-F(x))$.
\end{Theorem}
The idea for the proof is to use the central limit theorem for double arrays, see Section \nameref{Section:ProofsSzasz} for more details.
Note that as in the settings before, this result holds for all choices of $m$ with $m \rightarrow \infty$ without any restrictions.

We now take a closer look at the asymptotic behavior of $\hat{F}_{m,n}^S(x)-F(x)$, where the behavior of $m$ is restricted. With \autoref{Lemma:LorentzS}, it is easy to see that
\begin{align}
  \label{Eq:AsyBernsteinS}
  n^{1/2}\left(\hat{F}_{m,n}^S(x)-F(x)\right)&=n^{1/2}\left(\hat{F}_{m,n}^S(x)-S_m(F;x)\right)\nonumber\\
  &\enspace+m^{-1}n^{1/2}b^S(x)+o_x(m^{-1}n^{1/2}).
\end{align}
This leads directly to the following corollary, which is similar to Corollary 2 in \cite{leblancEstimatingDistributionFunctions2012} but on $(0,\infty)$.

\begin{Corollary}
  Let $m,n \rightarrow \infty$. Then, for $x\in(0,\infty)$ with $0<F(x)<1$, it holds that
  \begin{enumerate}
    \item[(a)] if $mn^{-1/2} \rightarrow \infty$, then
    \begin{equation*}
      n^{1/2}\left(\hat{F}_{m,n}^S(x)-F(x)\right) \xrightarrow{D} \mathcal{N}\left(0,\sigma^2(x)\right),
    \end{equation*}
    \item[(b)] if $mn^{-1/2} \rightarrow c$, where $c$ is a positive constant, then
    \begin{equation*}
      n^{1/2}\left(\hat{F}_{m,n}^S(x)-F(x)\right) \xrightarrow{D} \mathcal{N}\left(c^{-1}b^S(x),\sigma^2(x)\right),
    \end{equation*}
  \end{enumerate}
  where $\sigma^2(x)$ and $b^S(x)$ are defined in \autoref{Theorem:BiasVarS}.
\end{Corollary}


\subsection{Asymptotically Optimal $m$ with Respect to Mean-squared Error}

For the estimator $\hat{F}_{m,n}^S$, it is interesting to calculate the  mean-squared error (MSE)
\begin{equation*}
  \MSE\left[\hat{F}_{m,n}^S(x)\right]=\E\left[\left(\hat{F}_{m,n}^S(x)-F(x)\right)^2\right]
\end{equation*}
and the asymptotically optimal $m$ with respect to MSE. The MSE at $x=0$ is zero. For $(0,\infty)$, the next theorem shows the asymptotic MSE.

\begin{Theorem}
  \label{Theorem:MSES}
  The MSE of the Szasz estimator is of the form
  \begin{align}
    \label{Eq:MSEFmnS}
    \MSE\left[\hat{F}_{m,n}^S(x)\right] &= \Var\left[\hat{F}_{m,n}^S(x)\right]+\Bias\left[\hat{F}_{m,n}^S(x)\right]^2\nonumber\\
    &= n^{-1}\sigma^2(x)-m^{-1/2}n^{-1}V^S(x)+m^{-2}\left(b^S(x)\right)^2\nonumber\\
    &\enspace+o_x(m^{-2})+o_x(m^{-1/2}n^{-1})
  \end{align}
  for $x\in(0,\infty)$.
\end{Theorem}
\begin{proof}
  This follows directly from \autoref{Theorem:BiasVarS}.
\end{proof}

To calculate the optimal $m$ with respect to the MSE, one has to take the derivative with respect to $m$ of the above equation and set it to zero.
The next corollary, which is similar to Corollary 1 in \cite{leblancEstimatingDistributionFunctions2012}, follows.

\begin{Corollary}
  Assuming that $f(x)\neq 0$ and $f'(x) \neq 0$, the asymptotically optimal choice of $m$ for estimating $F(x)$ with respect to MSE is
  \begin{equation*}
    m_{opt}=n^{2/3}\left[\frac{4(b^S(x))^2}{V^S(x)}\right]^{2/3}.
  \end{equation*}
  Therefore, the associated MSE can be written as
  \begin{equation}
    \label{Eq:optMSESzasz}
    \MSE\left[\hat{F}_{m_{opt},n}^S(x)\right]
    =n^{-1}\sigma^2(x) -\frac{3}{4}n^{-4/3}\left[\frac{(V^S(x))^4}{4(b^S(x))^2}\right]^{1/3} +o_x(n^{-4/3})
  \end{equation}
  for $x \in (0,\infty)$, where $\sigma^2(x), b^S(x)$, and $V^S(x)$ are defined in \autoref{Theorem:BiasVarS}.
\end{Corollary}

In \cite{gawronskiDensityEstimationMeans1980}, it is stated that the optimal $m$ to estimate the density function with respect to the MSE is $O(n^{2/5})$. We just established that for the distribution function, the optimal rate is $O(n^{2/3})$. The same phenomenon that was first observed by  \cite{hjortNoteKernelDensity2001} for the kernel estimator and explained in \cite{leblancEstimatingDistributionFunctions2012} for the Bernstein estimator can be found here.

When using $m = O(n^{2/5})$ for the distribution estimation, it lies outside of any confidence band of $F$. This holds because of the fact that from $mn^{-2/5} \rightarrow c$ it follows that $mn^{-1/2} \rightarrow 0$. Together with $f'(x) \neq 0$ and \autoref{Eq:AsyBernsteinS}, it holds that
\begin{equation*}
  \P\left(n^{1/2}\left|\hat{F}_{m,n}^S(x)-F(x)\right|>\epsilon\right) \rightarrow 1
\end{equation*}
for all $\epsilon > 0$. This shows that for this choice of $m$, $\hat{F}_{m,n}^S(x)$ does not converge to a limiting distribution centered at $F(x)$ with proper rescaling. Therefore, $\hat{F}_{m,n}^S$ lies outside of any confidence band based on $F_n$ with probability going to one.


\subsection{Asymptotically Optimal $m$ with Respect to Mean-integrated Squared Error}

We now focus on the mean-integrated squared error (MISE). As we deal with an infinite integral, we use a non-negative weight function $\omega$. Here, the weight function is chosen as $\omega(x)=e^{-ax}f(x)$. Following \cite{altmanBandwidthSelectionKernel1995}, the MISE is then defined by
\begin{equation*}
  \MISE\left[\hat{F}_{m,n}^S\right]=\E\left[\int_0^{\infty}\left(\hat{F}_{m,n}^S(x)-F(x)\right)^2e^{-ax}f(x)\diff x\right].
\end{equation*}
Technically, $\MISE\left[\hat{F}_{m,n}^S\right]$ cannot be calculated by integrating the expression of $\MSE\left[\hat{F}_{m,n}^S\right]$ obtained in \autoref{Eq:MSEFmnS} as the asymptotic expressions depend on $x$.
The next theorem gives the asymptotic MISE of the Szasz operator and is similar to Theorem 3 in  \cite{leblancEstimatingDistributionFunctions2012}.

\begin{Theorem}
  \label{Theorem:MISES}
  We have
  \begin{equation*}
    \MISE\left[\hat{F}_{m,n}^S\right]=n^{-1}C_1^S-m^{-1/2}n^{-1}C_2^S+m^{-2}C_3^S+o(m^{-1/2}n^{-1})+o(m^{-2})
  \end{equation*}
  with
  \begin{align*}
    C_1^S=\int_0^{\infty}\sigma^2(x)e^{-ax}f(x)&\diff x \enspace \text{,} \enspace C_2^S=\int_0^{\infty}V^S(x)e^{-ax}f(x)\diff x \text{, and}\nonumber\\
    &C_3^S=\int_0^{\infty} (b^S(x))^2e^{-ax}f(x)\diff x.
  \end{align*}
  The definitions of $\sigma^2(x), b^S(x)$, and $V^S(x)$ can be found in \autoref{Theorem:BiasVarS}.
\end{Theorem}
\noindent For the proof, see Section \nameref{Section:ProofsSzasz}.

Very similar to Corollary 4 in \cite{leblancEstimatingDistributionFunctions2012}, the next corollary gives the asymptotically optimal $m$ for estimating $F$ with respect to MISE.
\begin{Corollary}
  \label{Corollary:MISESzasz}
  The asymptotically optimal $m$ for estimating $F$ with respect to MISE is
  \begin{equation*}
    m_{opt}=n^{2/3}\left[\frac{4C_3^S}{C_2^S}\right]^{2/3},
  \end{equation*}
  which leads to the optimal MISE
  \begin{equation}
    \label{Eq:optMISESzasz}
    \MISE\left[\hat{F}_{m_{opt},n}^S\right] =n^{-1}C_1^S -\frac{3}{4}n^{-4/3}\left[\frac{(C_2^S)^4}{4C_3^S}\right]^{1/3} +o(n^{-4/3}).
  \end{equation}
\end{Corollary}

If we compare the optimal MSE and optimal MISE of the Szasz estimator with those of the EDF, we observe the same behavior as for the Bernstein estimator. The second summand (including the minus sign ahead of it) in \autoref{Eq:optMSESzasz} and \autoref{Eq:optMISESzasz} is always negative so that the Szasz estimator seems to outperform the EDF. This is proven in the following.

\subsection{Asymptotic deficiency of the empirical distribution function}

We now measure the local and global performance of the Szasz estimator with the help of the deficiency. Let
\begin{align*}
  i_L^S(n,x)&=\text{min}\left\{k\in\N: \MSE[F_k(x)]\leq\MSE\left[\hat{F}_{m,n}^S(x)\right]\right\}, \enspace \text{and}\nonumber\\
  i_G^S(n)&=\text{min}\left\{k\in\N: \MISE[F_k]\leq\MISE\left[\hat{F}_{m,n}^S\right]\right\}
\end{align*}
be the local and global numbers of observations that $F_n$ needs to perform at least as well as $\hat{F}_{m,n}^S$. The next theorem deals with these quantities and is similar to Theorem 4 in \cite{leblancEstimatingDistributionFunctions2012}.

\noindent
\vbox{
\begin{Theorem}
  \label{Theorem:DeficiencyS}
  Let $x\in(0,\infty)$ and $m,n \rightarrow \infty$. If $mn^{-1/2}\rightarrow \infty$, then,
  \begin{equation*}
    i_L^S(n,x)=n[1+o_x(1)] \enspace \text{and} \enspace i_G^S(n)=n[1+o(1)].
  \end{equation*}
  In addition, the following statements are true.
  \begin{enumerate}
    \item[(a)] If $mn^{-2/3} \rightarrow \infty$ and $mn^{-2} \rightarrow 0$, then
    \begin{align*}
        i_L^S(n,x)-n&=m^{-1/2}n[\theta^S(x)+o_x(1)], \enspace \text{and}\\
        i_G^S(n)-n&=m^{-1/2}n[C_2^S/C_1^S+o(1)].
    \end{align*}
    \item[(b)] If $mn^{-2/3} \rightarrow c$, where $c$ is a positive constant, then
    \begin{align*}
        i_L^S(n,x)-n&=n^{2/3}[c^{-1/2}\theta^S(x)-c^{-2}\gamma^S(x)+o_x(1)], \enspace \text{and}\\
        i_G^S(n)-n&=n^{2/3}[c^{-1/2}C_2^S/C_1^S-c^{-2}C_3^S/C_1^S+o(1)],
    \end{align*}
  \end{enumerate}
  where
    \begin{equation*}
      \theta^S(x)=\frac{V^S(x)}{\sigma^2(x)} \enspace \text{and} \enspace \gamma^S(x)=\frac{(b^S(x))^2}{\sigma^2(x)}.
    \end{equation*}
  Here, $V^S(x), \sigma^2(x)$, and $b^S(x)$ are defined in \autoref{Theorem:BiasVarS} and $C_1^S, C_2^S$, and $C_3^S$ are defined in \autoref{Theorem:MISES}.
\end{Theorem}
\noindent For the proof, see Section \nameref{Section:ProofsSzasz}.}

This theorem shows under which conditions the Szasz estimator outperforms the EDF. The asymptotic deficiency goes to infinity as $n$ grows. This means that for increasing $n$, the number of extra observations also has to increase to infinity so that the EDF outperforms the Szasz estimator. Hence, the EDF is asymptotically deficient to the Szasz estimator.

It seems natural that one can also base the selection of an optimal $m$ on the deficiency. Indeed, maximizing the deficiency seems a good way to make sure that the Szasz estimator outperforms the EDF as much as possible.

\begin{Lemma}
   The optimal $m$ with respect to the global deficiency in the case $mn^{-2/3} \rightarrow c$ is of the same order as in \autoref{Corollary:MISESzasz}.
\end{Lemma}
\begin{proof}
  The proof follows arguments in \cite{leblancEstimatingDistributionFunctions2012}. In the case $mn^{-2/3} \rightarrow c$, the deficiency $i_G(n)-n$ is asymptotically positive only when $c > \left[\frac{C_3^S}{C_2^S}\right]^{2/3}=c^*$.
  Then, the optimal $c$ maximizing $g(c)=c^{-1/2}C_2^S/C_1^S-c^{-2}C_3^S/C_1^S$ is
  \begin{equation*}
    c_{opt}=\left[\frac{4C_3^S}{C_2^S}\right]^{2/3}=2^{4/3}c^*.
  \end{equation*}
  Hence, the optimal order of the Szasz estimator with respect to the deficiency satisfies
  \begin{equation*}
    m_{opt}n^{-2/3} \rightarrow c_{opt} \Leftrightarrow m_{opt}=n^{2/3}[c_{opt}+o(1)].
  \end{equation*}
\end{proof}

\section{Theoretical comparison}
\label{Section:Comparison}

\DoubleFigureHere{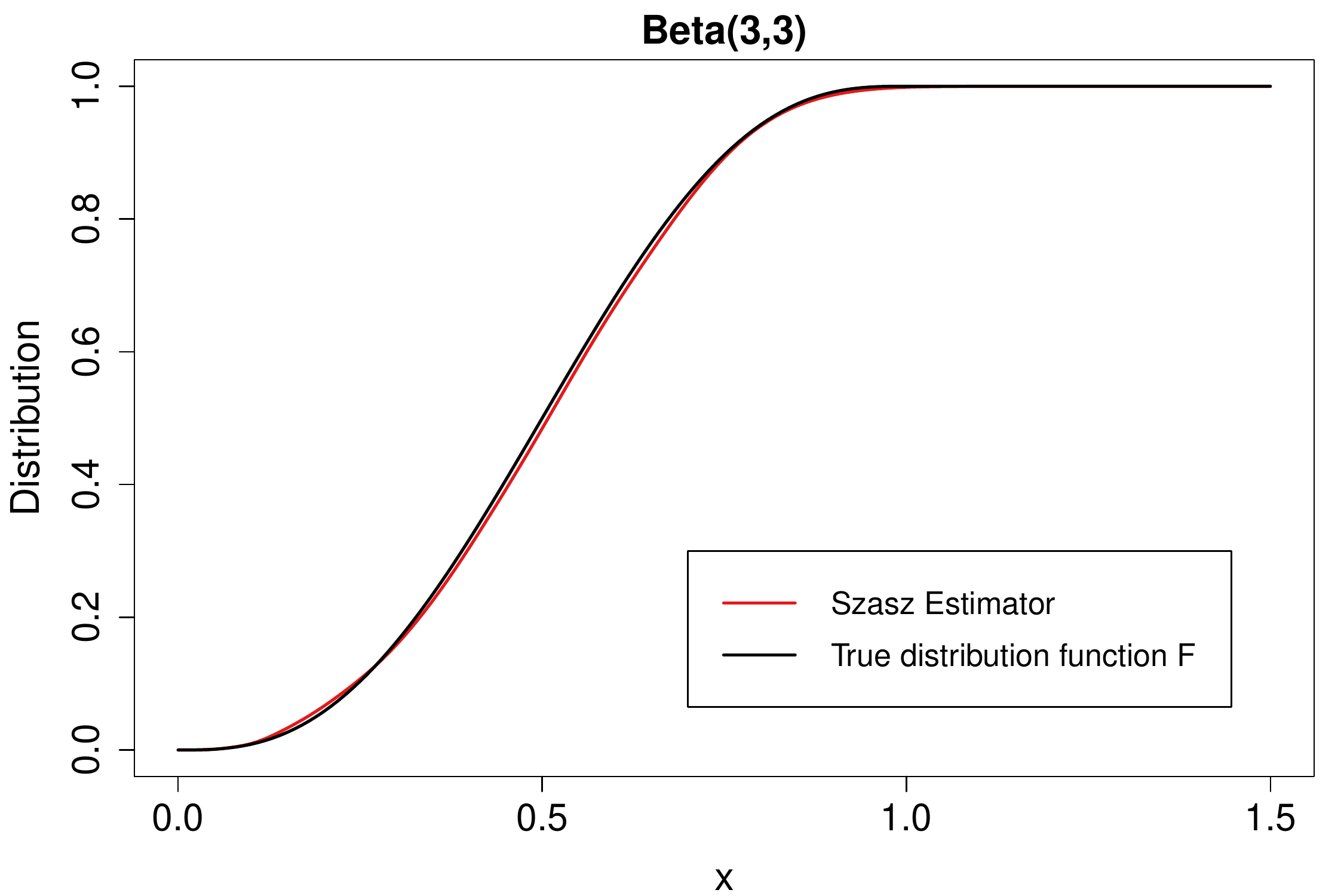}{The behavior of the Szasz estimator at $x=1$ for $n=500$.}

In the following, the properties that were derived in this paper for the Szasz estimator are compared to the different estimators defined in the introduction. The comparison can be found in Tables \ref{tab:1}-\ref{tab:4}. The assumptions in the third column of the first table have to be fulfilled for the theoretical results to hold. If there are extra assumptions for one specific result, they are written as a footnote. More details can be found in \cite{hanebeckNonparametricDistributionFunction2020}.

For the EDF, the properties mainly follow from famous theorems. The uniform, almost sure convergence follows from the Glivenko-Cantelli theorem while the asymptotic normality can be proven with the central limit theorem. The MSE can be found in \cite{lockhartBasicsNonparametricModels2013} and the other properties are easy to calculate.
For the kernel estimator, the asymptotic normality can be found in \cite{watsonHazardAnalysisII1964} and \cite{zhangEstimatingDistributionFunction2020}, while bias and variance can be found in 
\cite{kimBiasReducingTechnique2006}. The optimal MSE and MISE can be found in \cite{zhangEstimatingDistributionFunction2020}.
The properties for the Bernstein estimator mainly follow from \cite{leblancEstimatingDistributionFunctions2012}, where some results are using ideas from \cite{babuApplicationBernsteinPolynomials2002}.
The ideas and most of the proofs for the Hermite estimators can be found in \cite{stephanouSequentialQuantilesHermite2017} and \cite{stephanouPropertiesHermiteSeries2020a} for the estimator on the real half line and on the real line respectively.

The results on the asymptotic normality of the Hermite estimators are new. For the Hermite estimator on the real half line, the following theorem holds.
\begin{Theorem}
    \label{Theorem:NormalHalf}
    For $x \in (0, \infty)$ with $0<F(x)<1$ and if $f$ is differentiable in $x$, we obtain
    \begin{equation*}
        \sqrt{n}\left(\hat{F}_{N,n}^H(x)-\E\left[\hat{F}_{N,n}^H(x)\right]\right) \xrightarrow{D} \mathcal{N}\left(0,\sigma^2(x)\right),
    \end{equation*}
    for $n \rightarrow \infty$, where $\sigma^2(x)=F(x)(1-F(x))$.
\end{Theorem}
The proof can be found in the appendix.
For the theorem on the real line and the corresponding proof, we refer to \cite{hanebeckNonparametricDistributionFunction2020}.

It is important to always make sure that the situation fits to compare different estimators.
A comparison between the Bernstein estimator and the Szasz estimator for example only makes sense when the density function on $[0,1]$ can be continued to $[0,\infty)$ so that \autoref{Assumption:3} holds. Of course it is also possible to use the Szasz estimator for distributions where $F$ is continuous on $[0,\infty)$ and $f$ is not. Then, the theoretical results do not hold anymore but convergence is still given. But we know that the Bernstein estimator is always better as it has zero bias and variance for $x=1$, while the Szasz estimator has the continuous derivative
\begin{align*}
    \frac{d}{\diff x}\hat{F}_{m,n}^S(x)=m\sum_{k=0}^{\infty}\left[F_n\left(\frac{k+1}{m}\right)-F_n\left(\frac{k}{m}\right)\right]e^{-mx}\frac{(mx)^k}{k!}
\end{align*}
and cannot approximate a non-continuous function that well. This can be seen in \autoref{ComparisonSzaszBernstein}. It is obvious that the behavior of the Szasz estimator at $x=1$ of the $\text{Beta}(2,1)$-distribution is worse.

For the Hermite estimators the properties $f \in L_2$ and $\left(x-\frac{d}{\diff x}f\right)^rf \in L_2$ only have to hold on the considered interval. Hence, they can be used for smaller intervals than what they were designed for.

The EDF and the kernel distribution function estimator can be used on arbitrary intervals. However, note that the asymptotic results for the kernel estimator hold under the assumption that the density occupies $(-\infty,\infty)$. Hence, if the support is bounded, the results do not hold for the points close to the boundary. For an approach to improve this boundary behavior, see \cite{zhangEstimatingDistributionFunction2020} for example.

\subsection{Some Observations}
In the following, some important observations regarding the theoretical comparison are listed. It is notable that for the asymptotic order, $h=1/m$ for the Bernstein estimator is always replaced by $h^2$ for the Kernel estimator. Also, the results for the Szasz estimator are the same as for the Bernstein estimator with the exception that the orders are often not uniform.

There are some properties that some or all of the estimators have in common.
Regarding the deficiency, the Bernstein estimator, the kernel estimator, and the Szasz estimator all outperform the EDF with respect to MSE and MISE.
All of the estimators convergence a.s. uniformly to the true distribution function, and the asymptotic distribution of the scaled difference between estimator and the true value always coincide under different assumptions.

However, there are of course also many differences between the estimators that are addressed now.
For the Bernstein estimator and the Szasz estimator, the order of the bias is worse than that of the kernel estimator. The order of the Hermite estimator on the real half line depends on $x$.
This is not the case for the estimator on the real line. On the other hand, the order for the estimator on the real line is worse.

For the variance, the orders of the Bernstein estimator and the Szasz estimator are the same as for the EDF and the kernel estimator but are not uniform. The Hermite estimator on the real line is worse than the estimator for the real half line but uniform. Their orders are both worse than that of the other estimators.

The optimal rate of the MSE is $n^{-1}$ for the first four estimators in the table, two of them uniform and the others not. The rates of the Hermite estimators are worse but for $r \rightarrow \infty$, the rates also approach $n^{-1}$. This is very similar for the optimal rates of the MISE.

\begin{savenotes}
\begin{table}
  \centering
    \caption{Support of the estimators and assumptions}
    \label{tab:1}
    \begin{tabular}{|C{3cm}|C{3cm}|C{4.5cm}|}
        \hline
        \begin{tabular}{c}\makecell{\phantom{}\\ \phantom{}}\end{tabular} & Support & Assumptions
        \\\hline

        \begin{tabular}{c}\makecell{EDF}\end{tabular}
        & \makecell{Chosen Freely}
        &
        \\\hline

        \begin{tabular}{c}\makecell{Kernel}\end{tabular}
        & \makecell{Chosen Freely}
        & \makecell{Density $f$ exists, $f'$ exists\\ and is continuous}
        \\\hline

        \begin{tabular}{c}\makecell{Bernstein}\end{tabular}
        & \makecell{$[0,1]$}
        & \makecell{$F$ continuous, two\\ continuous and bounded\\ derivatives on $[0,1]$}
        \\\hline

        \begin{tabular}{c}\makecell{Szasz}\end{tabular}
        & \makecell{$[0,\infty)$}
        & \makecell{$F$ continuous, two\\ continuous and bounded\\ derivatives on $[0,\infty)$}
        \\\hline

        \begin{tabular}{c}\makecell{Hermite Half}\end{tabular}
        & \makecell{$[0,\infty)$}
        & \makecell{$f \in L_2$}
        \\\hline

        \begin{tabular}{c}\makecell{Hermite Full}\end{tabular}
        & \makecell{$(-\infty,\infty)$}
        & \makecell{$f \in L_2$}
        \\\hline
    \end{tabular}
\end{table}
\end{savenotes}

\begin{savenotes}
\begin{table}
  \centering
    \caption{Convergence behavior and asymptotic distribution of the estimators}
    \label{tab:2}
    \begin{tabular}{|C{3.0cm}|C{3.0cm}|C{6.5cm}|}
        \hline
        \begin{tabular}{c}\makecell{\phantom{}\\ \phantom{}}\end{tabular} & Convergence
        & \makecell{Asymptotic distribution:\\ $n^{1/2}(\hat{F}_{n}(x)-F(x)) \xrightarrow{D} \mathcal{N}\left(0,\sigma^2(x)\right)$ \footnote{$\hat{F}_n$ stands for all of the estimators, for $x: 0<F(x)<1$}}
        \\\hline

        \begin{tabular}{c}\makecell{EDF}\end{tabular}
        & \makecell{a.s. uniform}
        & 
        \\\hline

        \begin{tabular}{c}\makecell{Kernel}\end{tabular}
        & \makecell{a.s. uniform}
        & \makecell{For $h^{-2}n^{-1/2} \rightarrow \infty$} 
        \\\hline

        \begin{tabular}{c}\makecell{Bernstein}\end{tabular}
        & \makecell{a.s. uniform}
        & \makecell{For $mn^{-1/2} \rightarrow \infty$} 
        \\\hline

        \begin{tabular}{c}\makecell{Szasz}\end{tabular}
        & \makecell{a.s. uniform}
        & \makecell{For $mn^{-1/2} \rightarrow \infty$} 
        \\\hline

        \begin{tabular}{c}\makecell{Hermite Half}\end{tabular}
        & \makecell{a.s. uniform \footnote{For $\left(x-\frac{d}{\diff x}\right)^rf\in L_2, r\geq 1, \E\left[|X|^{s}\right]<\infty, s>\frac{8(r+1)}{3(2r+1)}, N \sim n^{\frac{2}{2r+1}}$}}
        & \makecell{For $N^{r/2-1/4}n^{-1/2} \rightarrow \infty$ \footnote{For $\left(x-\frac{d}{\diff x}\right)^rf\in L_2, r\geq 1, \E\left[|X|^{2/3}\right]<\infty$}} 
        \\\hline

        \begin{tabular}{c}\makecell{Hermite Full}\end{tabular}
        & \makecell{a.s. uniform \footnote{For $\left(x-\frac{d}{\diff x}\right)^rf\in L_2, r>2, \E\left[|X|^{s}\right]<\infty, s>\frac{8(r+1)}{3(2r+1)}, N \sim n^{\frac{2}{2r+1}}$}}
        & \makecell{For $N^{r/2-1}n^{-1/2} \rightarrow \infty$ \footnote{For $\left(x-\frac{d}{\diff x}\right)^rf\in L_2, r>2$}} 
        \\\hline
    \end{tabular}
\end{table}
\end{savenotes}

\begin{savenotes}
\begin{table}
  \centering
    \caption{Bias and variance of the estimators}
    \label{tab:3}
    \begin{tabular}{|C{3.0cm}|C{3.5cm}|C{4.5cm}|}
        \hline
            \makecell{\phantom{} \\ \phantom{}} & Bias & Variance\\\hline

            \begin{tabular}{c}\makecell{EDF}\end{tabular}
            & Unbiased
            & $O(n^{-1})$
            \\\hline

            \begin{tabular}{c}\makecell{Kernel}\end{tabular}
            & $o(h^2)$
            & $O(n^{-1})+O(h/n)$
            \\\hline

            \begin{tabular}{c}\makecell{Bernstein}\end{tabular}
            & \makecell{Zero at $\{0,1\}$ \\ $O(m^{-1})=O(h)$}
            & \makecell{Zero at $\{0,1\}$ \\ $O(n^{-1})+O_x(m^{-1/2}n^{-1})$}
            \\\hline

            \begin{tabular}{c}\makecell{Szasz}\end{tabular}
            & \makecell{Zero at $0$ \\ $O_x(m^{-1})=O_x(h)$}
            & \makecell{Zero at $0$ \\ $O(n^{-1})+O_x(m^{-1/2}n^{-1})$}
            \\\hline

            \begin{tabular}{c}\makecell{Hermite Half}\end{tabular}
            & \makecell{Zero at $0$ \\ $O_x\left(N^{-r/2+1/4}\right)$\footnote{\label{HalfBias}For $\left(x-\frac{d}{\diff x}\right)^rf\in L_2, r\geq 1, \E\left[|X|^{2/3}\right]<\infty$}}
            & \makecell{Zero at $0$ \\ $O_x(N^{3/2}/n)$\footnote{\label{HalfVariance}For $\E\left[|X|^{2/3}\right]<\infty$}}
            \\\hline

            \begin{tabular}{c}\makecell{Hermite Full}\end{tabular}
            & $O(N^{1-r/2})$\footnote{\label{FullBias}For $\left(x-\frac{d}{\diff x}\right)^rf\in L_2, r>2$}
            & $O(N^{5/2}/n)$\footnoteref{HalfVariance}
            \\\hline
        \end{tabular}
\end{table}
\end{savenotes}

\begin{savenotes}
\begin{table}
  \centering
    \caption{MSE and MISE of the estimators}
    \label{tab:4}
    \begin{tabular}{|C{6.5cm}|C{6.5cm}|}
        \hline
            MSE (all consistent) & MISE (all consistent)\\\hline

            $O(n^{-1})$
            & \begin{tabular}{c}$O(n^{-1})$\end{tabular}
            \\\hline

            \makecell{$O(n^{-1})+O\left(h^4\right)+O(h/n)$ \\ Optimal: $O(n^{-1})$}
            & \begin{tabular}{c}\makecell{$O(n^{-1})+O(h^4)+O(h/n)$ \\ Optimal: $O(n^{-1})$}\end{tabular}
            \\\hline

            \makecell{Zero at $\{0,1\}$ \\ $O(n^{-1})+O(m^{-2})+O_x(m^{-1/2}n^{-1})$ \\ Optimal: $O_x(n^{-1})$}
            & \begin{tabular}{c}\makecell{$O(n^{-1})+O(m^{-2})+O(m^{-1/2}n^{-1})$ \\ Optimal: $O(n^{-1})$}\end{tabular}
            \\\hline

            \makecell{Zero at $0$ \\ $O(n^{-1})+O_x(m^{-2})+O_x(m^{-1/2}n^{-1})$ \\ Optimal: $O_x(n^{-1})$}
            & \begin{tabular}{c}\makecell{$O(n^{-1})+O(m^{-2})+O(m^{-1/2}n^{-1})$ \\ Optimal: $O(n^{-1})$\footnote{Note that the MISE here is defined differently with weight function $e^{-ax}$}}\end{tabular}
            \\\hline

            \makecell{Zero at $0$ \\ $x\left[O\left(\frac{N^{1/2}}{n}\right)+O\left(N^{-r}\right)\right]$ \\ Optimal: $xO(n^{\frac{-2r}{2r+1}})$ \footnoteref{HalfBias}}
            & \begin{tabular}{c}\makecell{$\mu\left[O\left(\frac{N^{1/2}}{n}\right)+O\left(N^{-r}\right)\right]$ \\ Optimal: $\mu O(n^{-\frac{2r}{2r+1}})$\footnote{For $\left(x-\frac{d}{\diff x}\right)^rf\in L_2, r\geq 1, \mu=\int_0^{\infty} xf(x)dx<\infty$}}\end{tabular}
            \\\hline

            \makecell{$O\left(\frac{N^{5/2}}{n}\right)+O\left(N^{-r+2}\right)$ \\ Optimal: $O(n^{-\frac{2(r-2)}{2r+1}})$\footnote{\label{FullMSE}For $\left(x-\frac{d}{\diff x}\right)^rf\in L_2, r>2, \E\left[|X|^{2/3}\right]<\infty$}}
            & \begin{tabular}{c}\makecell{$O\left(\frac{N^{5/2}}{n}\right)+O\left(N^{-r+2}\right)$ \\ Optimal: $O\left(n^{-\frac{2(r-2)}{2r+1}}\right)$\footnoteref{FullMSE}}\end{tabular}
            \\\hline
        \end{tabular}
\end{table}
\end{savenotes}

\section{Simulation}
\label{Section:Simulation}

In this section, the different estimators are compared in a simulation study with respect to the MISE quality.

For the kernel distribution function estimator, the Gaussian kernel is chosen, i.e. $F_{h,n}(x)=\frac{1}{n}\sum_{i=1}^n\Phi\left(\frac{x-X_i}{h}\right)$, where $\Phi$ is the standard normal distribution function.

The simulation consists of two parts. In the first part, the estimators are compared by their MISE on $[0,\infty)$ with respect to
\begin{equation*}
    \MISE\left[\hat{F}_n\right]
    =\E\left[\int_0^{\infty}\left(\hat{F}_n(x)-F(x)\right)^2\cdot f(x)\diff x\right],
\end{equation*}
where $\hat{F}_n$ can be any of the considered estimators. In the second part, the asymptotic normality of the estimators is illustrated for one distribution.
The details for each part as well as the most important results are explained later.

All of the estimators except for the EDF have a parameter in addition to $n$. For these estimators, the MISE is calculated for a range of the parameters, which are given in \autoref{Table_Parameters}.

\begin{table}
    \centering
    \caption{The range of the respective parameters.}
    \label{Table_Parameters}
    \begin{tabular}{|c|c|c|}
        \hline
        Estimator & Abbr. & Parameters \\ \hline\hline
        EDF & $F_n$ & - \\ \hline
        Kernel & $F_{h,n}$ & $h=i/1000$, $i \in [2,200]$ \\ \hline
        Szasz & $\hat{F}_{m,n}^S$ & $m \in [2,200]$ \\ \hline
        Hermite Half & $\hat{F}_{N,n}^H$ & $N \in [2,60]$ \\ \hline
    \end{tabular}
\end{table}

We obtain a vector of MISE values for each estimator. Searching for the minimum value in this vector provides the minimal MISE and the respective optimal parameter.

Note that a selection of $m$ could be based on $m_{opt }$, defined in \autoref{Corollary:MISESzasz}, using ideas from automatic bandwith selection in kernel density estimation. Rule-of-thumb selectors replace the unknown density and distribution function with a reference distribution, for example the exponential distribution in our case. For plug-in selectors, the unknown quantities are estimated using pilot values of $m$. However, the analysis of such proposals is clearly far beyond the scope of this work.

Every MISE is calculated by a Monte-Carlo simulation with $M=10\,000$ repetitions. To be specific, let
\begin{equation*}
    \text{ISE}\left[\hat{F}_n\right]=\int_0^\infty\left[\hat{F}_n(x)-F(x)\right]^2 \cdot f(x) \diff x,
\end{equation*}
and with $M$ pseudo-random samples, the estimate of the MISE is calculated by
\begin{equation*}
    \MISE\left[\hat{F}_n\right]\simeq \frac{1}{M}\sum_{i=1}^M \text{ISE}_i[\hat{F}],
\end{equation*}
where $\text{ISE}_i$ is the integrated squared error calculated from the $i$th randomly generated sample.

For the Hermite estimator, the standardization explained in \cite{hanebeckNonparametricDistributionFunction2020} is used. In this simulation, we do not estimate the mean $\mu$ and the standard deviation $\sigma$ as we already know the true parameters.

\subsection{Comparison of the estimators}

For comparison, the exponential distribution with parameter $\lambda=2$ is chosen as well as three different Weibull mixture distributions. The bi- and trimodal mixtures that are considered are:
\begin{align*}
  \text{Weibull 1: } & 0.5 \cdot Weibull(1,1)+0.5 \cdot Weibull(4,4)\\
  \text{Weibull 2: } & 0.5 \cdot Weibull(3/2, 3/2)+0.5 \cdot Weibull(5,5)\\
  \text{Weibull 3: } & 0.35 \cdot Weibull(3/2,3/2)+0.35\cdot Weibull(4.5,4.5)+0.3 \cdot Weibull(8,8).
\end{align*}
Their densities are displayed in \autoref{WeibullDensities}.
Of course, the comparison of the estimators on $[0,\infty)$ means that the Bernstein estimator cannot be used. Likewise, we omit the Hermite estimator on the real line.

For the exponential distribution, the different sample sizes that are used are $n=20,50, 100$, and $500$. For the Weibull distributions, only $n=50$ and $n=200$ are considered.

\SingleFigureHere{0.87\textwidth}{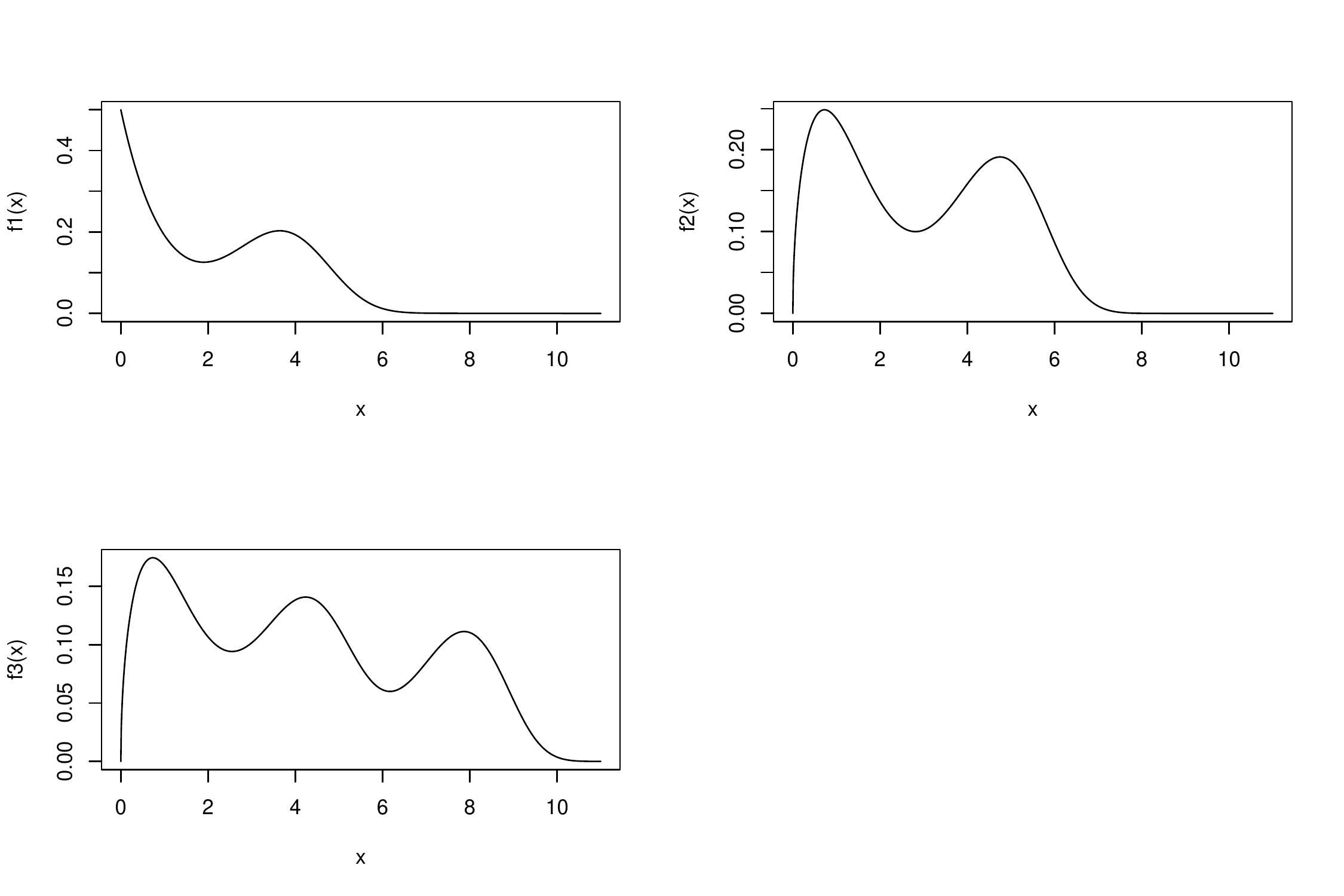}{Density plots of the three Weibull mixtures.}{WeibullDensities}

\DoubleFigureHere{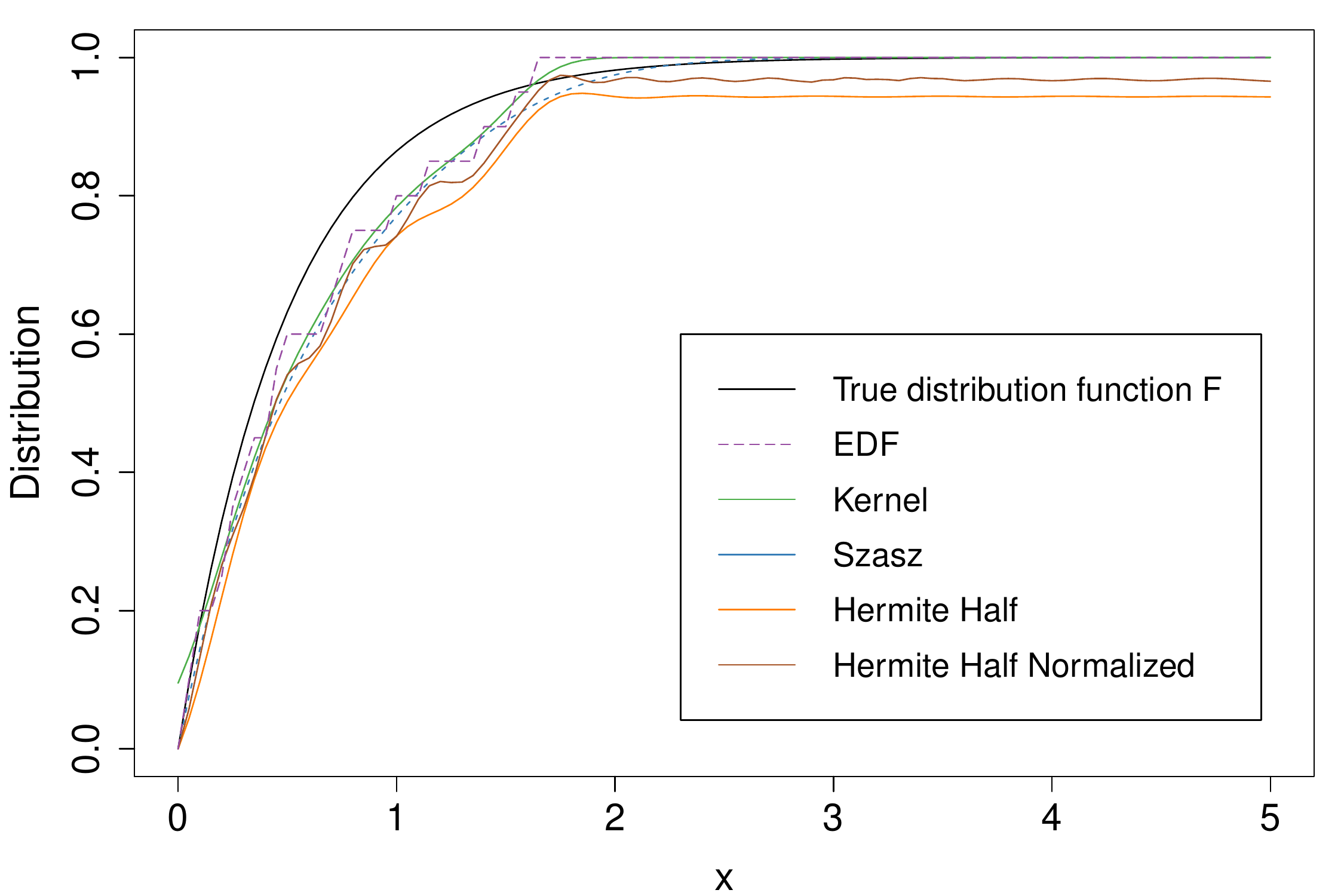}{Plot of the considered estimators for $n=20$ and $n=500$.}
An example of the different estimators can be seen in \autoref{HalfPlot} for $n=20$ and $n=500$. It is obvious that the Hermite estimators do not approach one, which is due to the truncation.

\DoubleFigureHere{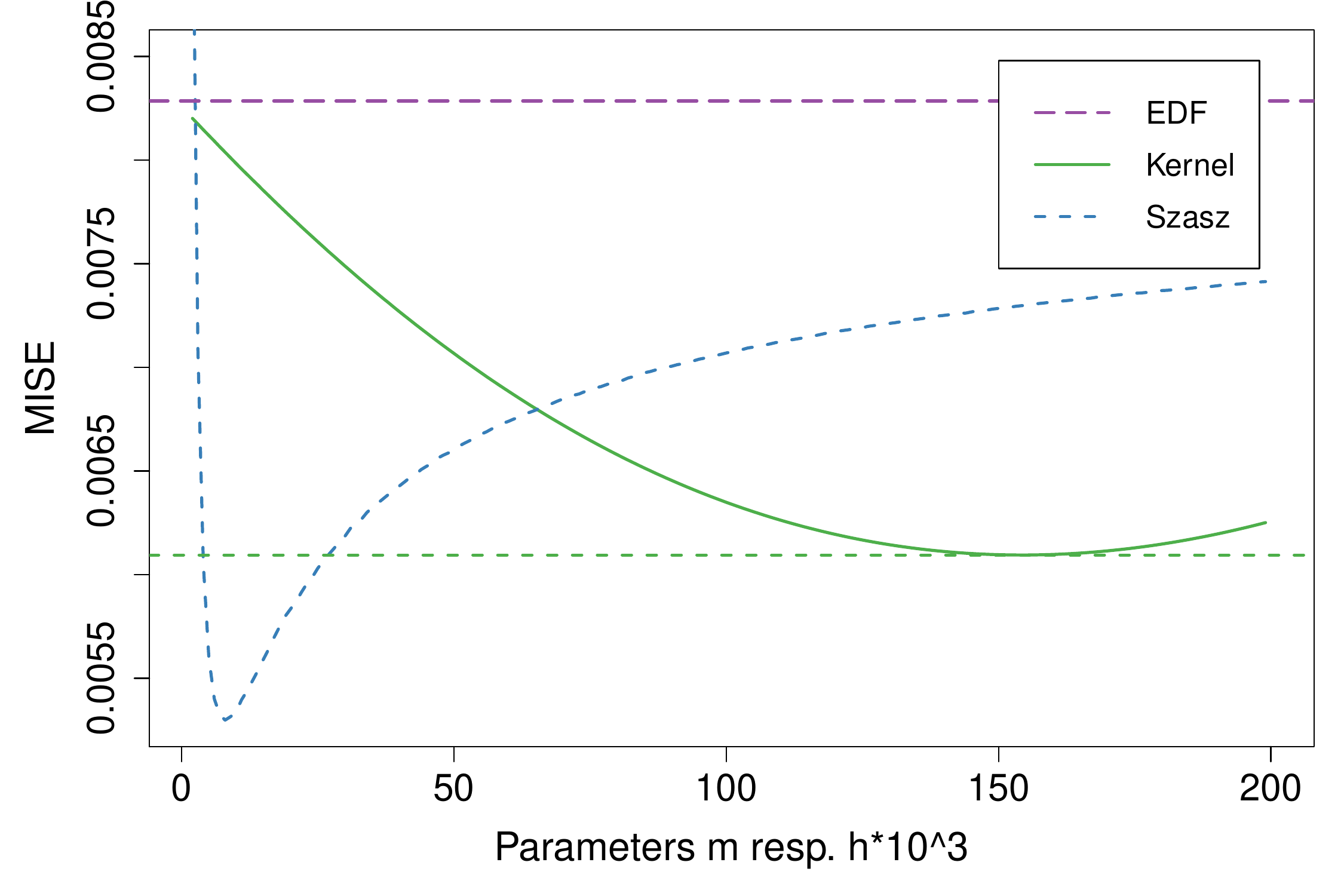}{MISE over the respective parameters in $[2,200]$ for $n=20$ and $n=500$ in the case of the exponential distribution.}
The Szasz estimator designed for the $[0,\infty)$-interval behaves best with respect to MISE. This can be seen in \autoref{HalfMISE1} for the exponential distribution. The minimal MISE-value of the Szasz estimator is always lower than that of the other estimators, also for the cases $n=50$ and $n=100$ that are not shown here.

\DoubleFigureHere{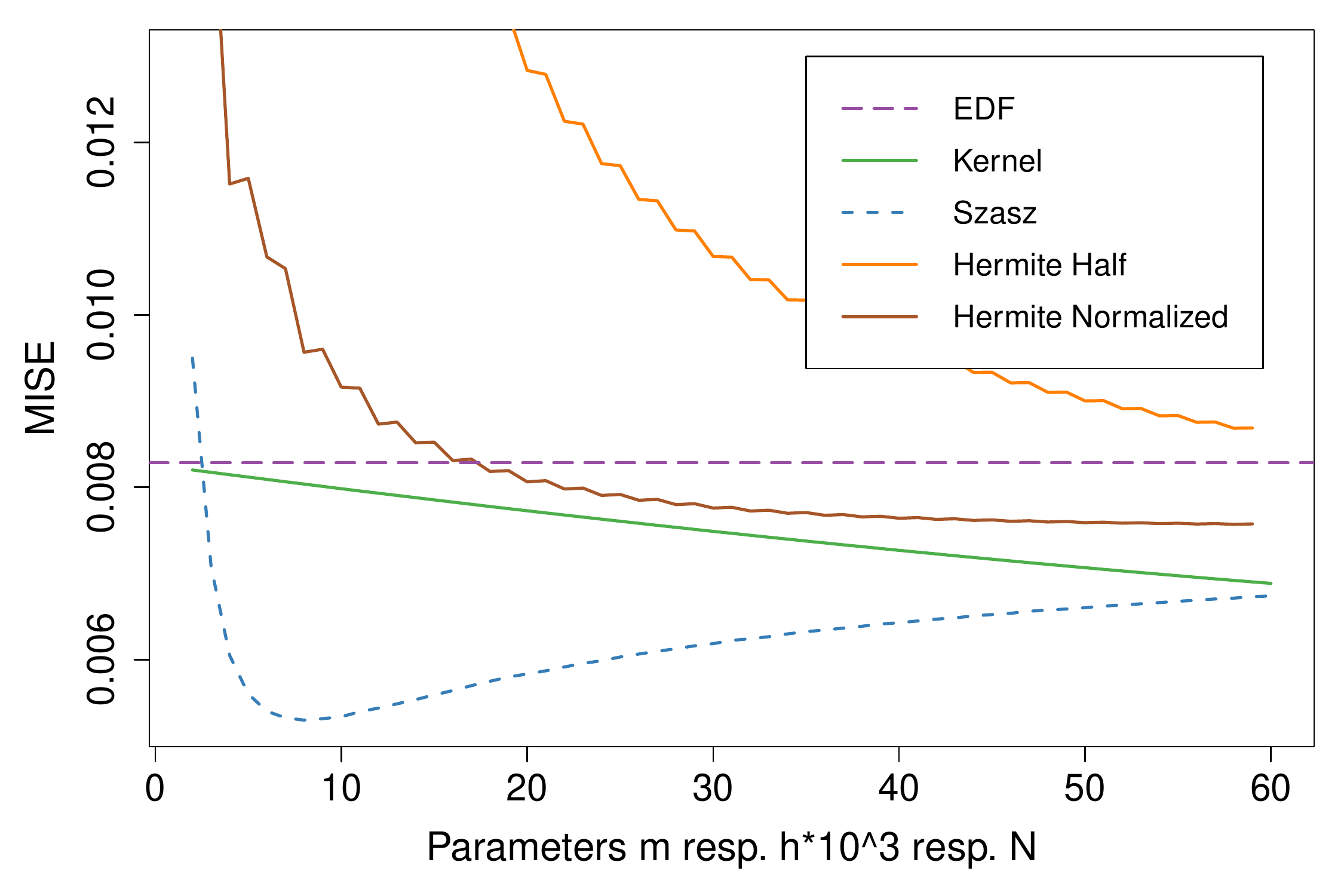}{MISE over the respective parameters in $[2,60]$ for $n=20$ and $n=500$ in the case of the exponential distribution.}

\autoref{HalfMISE2} makes clear that the standardization of the Hermite estimator 
yields a clear improvement over the nonstandardized estimator in the case of the exponential distribution, even for small sample sizes.

\begin{table}[ht]
  \centering
  \caption{The $\MISE \cdot 10^{-3}$-values for the interval $[0,\infty)$.}
  \label{Table_Second_Part}
  \begin{tabular}{|c|c|ccccc|}
    \hline
     & n & EDF & Kernel & Szasz & Hermite Half & Hermite Norm. \\ 
    \hline
    Exponential(2) & 20 & 8.29 & 6.09 & 5.3 & 8.68 & 7.57 \\ 
     & 50 & 3.3 & 2.71 & 2.41 & 5.61 & 3.58 \\ 
     & 100 & 1.68 & 1.47 & 1.32 & 4.6 & 2.26 \\ 
     & 500 & 0.34 & 0.32 & 0.3 & 3.73 & 1.15 \\ 
    \hline
    Weibull 1 & 50 & 3.32 & 2.92 & 2.55 & 3.26 & 3.45 \\ 
     & 200 & 0.83 & 0.76 & 0.71 & 0.99 & 1.33 \\ 
    \hline
    Weibull 2 & 50 & 3.32 & 2.96 & 2.59 & 3.08 & 2.76 \\ 
     & 200 & 0.83 & 0.75 & 0.72 & 0.79 & 0.79 \\
    \hline
    Weibull 3 & 50 & 3.36 & 3.11 & 2.55 & 3.26 & 2.91 \\ 
     & 200 & 0.83 & 0.77 & 0.73 & 0.81 & 0.8  \\
    \hline
  \end{tabular}
\end{table}

\autoref{Table_Second_Part} shows all the $\MISE \cdot 10^{-3}$-numbers of the optimal MISE for the considered estimators. The properties explained above for the exponential distribution can be found here as well. In the case of the Weibull distribution, the Szasz estimator also behaves the best in all cases.

\newpage
\subsection{Illustration of the Asymptotic Normality}
The goal here is to illustrate the asymptotic normality
\begin{equation*}
\sqrt{n}\left(\hat{F}_{n}(x)-F(x)\right) \xrightarrow{D} \mathcal{N}\left(0,\sigma^2(x)\right)
\end{equation*}
of the different estimators, where $\hat{F}_{n}$ can be any of the estimators. The expression can be rewritten as
\begin{equation*}
\hat{F}_n(x) \sim \mathcal{AN}\left(F(x),\frac{\sigma^2(x)}{n}\right).
\end{equation*}
This representation is used in the plots below for a $\Beta(3,3)$-distribution in the point $x=0.4$ for $n=500$. The value is $F(0.4)=0.32$.
In \autoref{Asymptotic}, the result can be seen. The red line in the plot shows the distribution function of the normal distribution. Furthermore, the histogram of the value $p=\hat{F}_n(0.4)$ is illustrated. The parameters used for the estimators are derived from the optimal parameters calculated in the simulation.

\FourFigureVeranschaulichungHermite{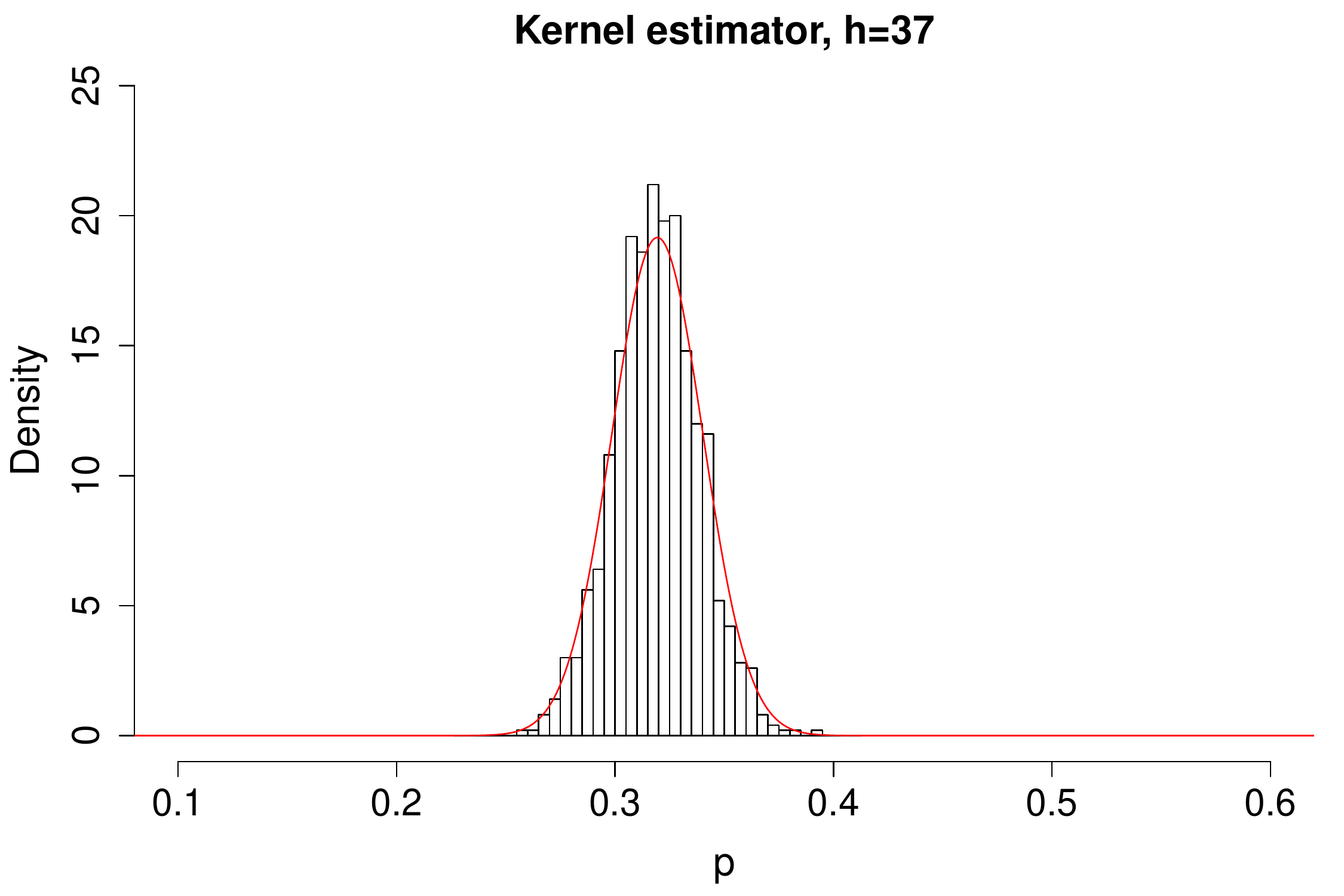}{Illustration of the asymptotic normal distribution.}

\newpage
\section{Conclusions}
\label{Section:Conclusions}

Surprisingly, there is not much literature on nonparametric smooth distribution function estimators especially tailored to distributions on the positive real half line. This important case occurs in many applications where the data can only be positive but does not have an upper bound, such as prices, losses, biometric data and much more.

In this article, we have introduced an estimator for distribution functions on $[0,\infty)$ based on Szasz-Mirakyan operators. We have shown that the Szasz estimator compares very well with other important estimators such as the kernel estimator in theoretical comparisons and in a simulation. Especially on the matching interval $[0,\infty)$, the simulation study shows a clear advantage of the Szasz estimator with respect to the MISE quality.

\section*{Acknowledgements}
The authors are grateful to Frédéric Ouimet for pointing out an error in a previous version of \autoref{Lemma:EigenschaftenLRS}, for helpful discussions and for sharing his preprint \cite{ouimetLocalLimitTheorem2020}. 

\section*{Appendix} 

\subsection*{Limit Theorem}

The following theorem can be found in \cite{ouimetLocalLimitTheorem2020}. He pointed out a mistake in the paper of \cite{leblancEstimatingDistributionFunctions2012} which also has an impact on this paper. The asymptotic behavior of $R_{1,m}^S$ in \autoref{Lemma:EigenschaftenLRS} has been corrected compared to Lemma 3 in \cite{hanebeckSmoothDistributionFunction2020}, arXiv v.1. This results in a slightly different definition of $V^S$ defined in \autoref{Theorem:BiasVarS}.
\begin{Theorem}
  \label{Lemma:LimitTheorem}
  We define
  \begin{equation*}
    V_{k,m}(x)=\frac{(mx)^k}{k!}e^{-mx}, \; \phi(x)=\frac{1}{\sqrt{2\pi}}e^{-x^2/2}, \; and \; \delta_k=\frac{k-mx}{\sqrt{mx}}.
  \end{equation*}
  Pick any $\eta \in (0,1)$. Then, we have uniformly for $k \in \N_0$ with $\left|\frac{\delta_k}{\sqrt{mx}}\right|\leq \eta$ that
  \begin{align*}
    \frac{V_{k,m}(x)}{\frac{1}{\sqrt{mx}}\phi(\delta_k)}&=1+m^{-1/2}\frac{1}{\sqrt{x}}\left(\frac{1}{6}\delta_k^3-\frac{1}{2}\delta_k\right)\\
    &+m^{-1}\frac{1}{x}\left(\frac{1}{72}\delta_k^6-\frac{1}{6}\delta_k^4+\frac{3}{8}\delta_k^2-\frac{1}{12}\right)+O_{x,\eta}\left(\frac{|1+\delta_k|^9}{m^{3/2}}\right)
  \end{align*}
  as $n \rightarrow \infty$.
\end{Theorem}

\subsection*{Properties of $V_{k,m}$}
  We now present various properties of $V_{k,m}$ that are needed for the proofs. The following lemma and its proof are similar to Lemma 2 and Lemma 3 in \cite{leblancEstimatingDistributionFunctions2012}. As mentioned before, parts (e) and (h) take suggestions in \cite{ouimetLocalLimitTheorem2020} into account. The proofs for these parts are adjusted accordingly.
\begin{Lemma}
    \label{Lemma:EigenschaftenLRS}
    Define
    \begin{equation*}
        L_m^S(x)=\sum_{k=0}^{\infty}V_{k,m}^2(x),
    \end{equation*}
    \begin{equation*}
        R^S_{j,m}(x)=m^{-j}\mathop{\sum\sum}_{0\leq k < l \leq \infty}(k-mx)^jV_{k,m}(x)V_{l,m}(x)
    \end{equation*}
    and
    \begin{equation*}
      \tilde{R}_{1,m}^S(x)=m^{1/2}\sum_{k,l=0}^{\infty}\left(\frac{k\wedge l}{m}-x\right)V_{k,m}(x)V_{l,m}(x)
    \end{equation*}
    for $j\in \{0,1,2\}$, and $V_{k,m}(x)=e^{-mx}\frac{(mx)^k}{k!}$. It trivially holds that $0 \leq L_m^S(x) \leq 1$ for $x\in[0,\infty)$.
    In addition, the following properties hold.
    Let $g$ be a continuous and bounded function on $[0,\infty)$. This leads to
    \begin{enumerate}
    \item[(a)] $L_m^S(0)=1$ and $\displaystyle\lim_{x\rightarrow \infty} L_m^S(x)=0$,
    \item[(b)] $R_{j,m}^S(0)=0$ for $j\in\{0,1,2\}$,
    \item[(c)] $0 \leq R_{2,m}^S(x) \leq \frac{x}{m} \enspace \text{for} \enspace x \in (0,\infty)$,
    \item[(d)] $L_m^S(x)=m^{-1/2}\left[(4\pi x)^{-1/2}+o_x(1)\right] \enspace \text{for} \enspace x\in (0,\infty)$,
    \item[(e)] $\tilde{R}_{1,m}^S(x)= -\sqrt{\frac{x}{\pi}}+o_x(1) \enspace \text{for} \enspace x\in (0,\infty)$ and $R_{1,m}^S(x)=m^{-1/2}\left[-\sqrt{\frac{x}{4\pi}}+o_x(1)\right]$,
    \item[(f)] $m^{1/2} \displaystyle\int_0^{\infty} L_m^S(x)e^{-ax}\diff x
        =\displaystyle\int_0^{\infty}(4\pi x)^{-1/2}e^{-ax}\diff x+o(1)=\frac{1}{2\sqrt{a}}+o(1)$ for $a \in (0,\infty)$,
    \item[(g)] $m^{1/2} \displaystyle\int_0^{\infty} x L_m^S(x)e^{-ax}\diff x
        =\displaystyle\int_0^{\infty}x^{1/2}(4\pi)^{-1/2}e^{-ax}\diff x+o(1)=\frac{1}{4a^{3/2}}+o(1)$ for $a \in (0,\infty)$,
    \item[(h)] $m^{1/2} \displaystyle\int_0^{\infty} g(x)R_{1,m}^S(x)e^{-ax}\diff x
    = -\displaystyle\int_0^{\infty} g(x)\frac{\sqrt{x}}{\sqrt{4\pi}}e^{-ax}\diff x+o(1)$ for $a \in (0,\infty)$ \\ and $\displaystyle\int_0^{\infty} g(x)\tilde{R}_{1,m}^S(x)e^{-ax}\diff x
    = -\displaystyle\int_0^{\infty} g(x)\frac{\sqrt{x}}{\sqrt{\pi}}e^{-ax}\diff x+o(1)$.
    \end{enumerate}
\end{Lemma}

\begin{proof}
  \begin{enumerate}
    \item[(a)]
    $L_m^S(0)=1$ is clear. Using the mode of the poisson distribution it holds for the limit that
    \begin{equation*}
        \lim_{x \rightarrow \infty} L_m^S(x)  \leq \lim_{x \rightarrow \infty} \max_k V_{k,m} \sum_{k=0}^{\infty} V_{k,m}  =  \lim_{x \rightarrow \infty} P(Y = \lfloor mx \rfloor)) = 0,
    \end{equation*}
    where $Y \sim \Poi(mx)$.

    \item[(b)]
    $R_{j,m}^S(0)=0$ holds trivially.

    \item[(c)]
    The non-negativity is clear. For the other inequality, it holds that
    \begin{align*}
        R_{2,m}^S(x) &\leq m^{-2}\sum_{k=0}^{\infty}\sum_{l=0}^{\infty}(k-mx)^2V_{k,m}(x)V_{l,m}(x)\nonumber\\
        &=m^{-2}\sum_{k=0}^{\infty}(k-mx)^2V_{k,m}(x)=m^{-2}\Var[Y]=\frac{x}{m},
    \end{align*}
    where $Y \sim \Poi(mx)$.

    \item[(d)]
    Let $U_i,W_j, i,j\in\{1,...,m\}$, be i.i.d. random variables with distribution $\Poi(x)$, hence,
    \begin{equation*}
        \P(U_1=k)=e^{-x}\frac{x^k}{k!}.
    \end{equation*}
    Define $R_i=(U_i-W_i)/\sqrt{2x}$.
    Then, we know that $\E[R_i]=0, \Var[R_i]=1$ and $R_i$ has a lattice distribution with span
        $h=1/\sqrt{2x}$.
    Note that with the independence it holds that
    \begin{align*}
        \P\left(\sum_{i=1}^{m}R_i=0\right)
        &=\P\left(\sum_{i=1}^{m}U_i=\sum_{i=1}^{m}W_i\right)\\
        &=\sum_{k=0}^{\infty} \P\left(\sum_{i=1}^{m}U_i=\sum_{i=1}^{m}W_i=k\right) \\
        &=\sum_{k=0}^{\infty}V_{k,m}^2(x).
    \end{align*}
    With Theorem 3 on p. 517 in \cite{fellerIntroductionProbabilityTheory1965}, we get that
    \begin{equation*}
        \frac{\sqrt{m}}{h}\sum_{k=0}^{\infty}V_{k,m}^2(x)-\frac{1}{\sqrt{2\pi}}\rightarrow 0
    \end{equation*}
    and it follows that
    \begin{equation*}
        \sqrt{4\pi mx}\sum_{k=0}^{\infty}V_{k,m}^2(x) \rightarrow 1,
    \end{equation*}
    from which the claim follows.

    \item[(e)] This proof is a part of the proof of Lemma 3.6 in \cite{ouimetLocalLimitTheorem2020}. We consider the decomposition
    \begin{equation}
      \label{Eq:Decomposition}
      \tilde{R}_{1,m}^S(x)=2m^{1/2}\sum_{0\leq k < l < \infty}\left(\frac{k}{m}-x\right)V_{k,m}(x)V_{l,m}(x)+m^{1/2}\sum_{k=0}^{\infty}\left(\frac{k}{m}-x\right)V_{k,m}^2.
    \end{equation}
    The second term on the right-hand side of Eq.(\ref{Eq:Decomposition}) is negligible as we know with the Cauchy-Schwarz inequality that
    \begin{align*}
      \sum_{k=0}^{\infty}\left(\frac{k}{m}-x\right)V_{k,m}^2(x)&\leq \left[\sum_{k=0}^{\infty}\left(\frac{k}{m}-x\right)^2V_{k,m}(x)\right]^{1/2}\left[\sum_{k=0}^m V_{k,m}^3(x)\right]^{1/2}\nonumber\\
      &\leq \left[\frac{T_{2,m}^S}{m^2}L_m^S(x)\right]^{1/2}\leq \left[\frac{x}{m}L_m^S(x)\right]^{1/2}\nonumber\\
      &\leq \left[\frac{x}{m}m^{-1/2}\left[(4\pi x)^{-1/2}+o_x(1)\right]\right]^{1/2}=o_x(m^{-3/4}),
    \end{align*}
    where $T_{2,m}^S(x)=\displaystyle\sum_{k=0}^{\infty}\left(k-mx\right)^2V_{k,m}(x)=mx$ for $x\in[0,\infty)$.
    For the first term on the right-hand side of Eq.(\ref{Eq:Decomposition}), we use the local limit theorem \autoref{Lemma:LimitTheorem} and integration by parts.
    Let $\phi_{\sigma^2}$ denote the density function of the $\mathcal{N}(0,\sigma^2)$ distribution. Then
    \begin{align*}
      \tilde{R}_{1,m}^S(x)&=2x\int_{-\infty}^{\infty}\frac{z}{x}\phi_x(z)\int_z^{\infty}\phi_x(y)dydz+o_x(1)\\
      &=2x\left[0-\int_{-\infty}^{\infty}\phi_x^2(z)dz\right]+o_x(1)\\
      &=\frac{-2x}{\sqrt{4\pi x}}\int_{-\infty}^{\infty}\phi_{\frac{1}{2}x}(z)dz+o_x(1)\\
      &=-\sqrt{\frac{x}{\pi}}+o_x(1).
    \end{align*}
    This leads to
    \begin{align*}
      R_{1,m}^S(x)=m^{-1/2}\left[-\sqrt{\frac{x}{4\pi}}+o_x(1)\right].
    \end{align*}

    \item[(f)]
    The goal is to calculate
    \begin{align*}
        m^{1/2} \int_0^{\infty} L_m^S(x)e^{-ax}\diff x&=m^{1/2} \int_0^{\infty} \sum_{k=0}^{\infty} V_{k,m}^2(x) e^{-ax}\diff x\nonumber\\
        &=m^{1/2} \sum_{k=0}^{\infty}\int_0^{\infty} V_{k,m}^2(x) e^{-ax}\diff x.
    \end{align*}
    For the integral we know that
    \begin{align*}
        \int_0^{\infty} V_{k,m}^2(x) e^{-ax}\diff x&=\int_0^{\infty} \left(e^{-mx}\frac{(mx)^k}{k!}\right)^2 e^{-ax}\diff x\nonumber\\
        &=\frac{m^{2k}}{(k!)^2} \int_0^{\infty}x^{2k}e^{-(2m+a)x}\diff x\nonumber\\
        &=\frac{m^{2k}}{(k!)^2(2m+a)^{2k+1}} \int_0^{\infty}y^{2k}e^{-y}dy\nonumber\\
        &=\frac{m^{2k}}{(k!)^2(2m+a)^{2k+1}}\Gamma(2k+1).
    \end{align*}
    Calculating the sum leads to
    \begin{align*}
        &m^{1/2} \sum_{k=0}^{\infty} \frac{m^{2k}}{(k!)^2(2m+a)^{2k+1}}\Gamma(2k+1)\nonumber\\
        &=m^{1/2} \sum_{k=0}^{\infty} \left(\frac{1}{2m+a}\right)^{2k+1}m^{2k}\binom{2k}{k}\nonumber\\
        &=\sqrt{\frac{m}{a(a+4m)}}\nonumber\\
        &=\frac{1}{2\sqrt{a}}+\frac{1}{\sqrt{a}}\left[\sqrt{\frac{m}{a+4m}}-\frac{1}{2}\right]=\frac{1}{2\sqrt{a}}+o(1).
    \end{align*}
    It holds that
    \begin{equation*}
        \int_0^{\infty}(4\pi x)^{-1/2}e^{-ax}\diff x=\frac{1}{2\sqrt{a}}
    \end{equation*}
    and hence,
    \begin{equation*}
        m^{1/2} \int_0^{\infty} L_m^S(x)e^{-ax}\diff x=\frac{1}{2\sqrt{a}}+o(1)=\int_0^{\infty}(4\pi x)^{-1/2}e^{-ax}\diff x+o(1).
    \end{equation*}

    \item[(g)]
    Similar to (f), we get
    \begin{equation*}
        \int_0^{\infty} xV_{k,m}^2(x) e^{-ax}\diff x=\frac{m^{2k}}{(k!)^2(2m+a)^{2k+2}}\Gamma(2k+2),
    \end{equation*}
    leading to
    \begin{align*}
        &m^{1/2} \sum_{k=0}^{\infty} \frac{m^{2k}}{(k!)^2(2m+a)^{2k+2}}\Gamma(2k+2)=\frac{\sqrt{m}(a+2m)}{(a(a+4m))^{3/2}}\nonumber\\
        &=\frac{1}{4a^{3/2}}+\frac{1}{a^{3/2}}\left[\frac{\sqrt{m}(a+2m)}{(a+4m)^{3/2}}-\frac{1}{4}\right]=\frac{1}{4a^{3/2}}+o(1).
    \end{align*}

    \item[(h)]
    Define $G_m^S(x)=m^{1/2}R_{1,m}^S(x)e^{-ax}$ and $G^S(x)=-\frac{\sqrt{x}}{\sqrt{4\pi}}e^{-ax}$. Then with part (e) we know that  $G_m^S(x) \xrightarrow{m \rightarrow \infty} G^S(x)$.

    Note that
    \begin{align*}
        R_{1,m}(x)&=m^{-1}e^{-2mx}\mathop{\sum\sum}_{0\leq k < l \leq \infty}(k-mx)\frac{(mx)^{k+l}}{k! l!}\nonumber\\
        &=m^{-1}e^{-2mx}\sum_{k=0}^{\infty}(k-mx)\frac{(mx)^k}{k!}\left[\sum_{l=k+1}^{\infty}\frac{(mx)^{l}}{l!}\right]\nonumber\\
        &=m^{-1}e^{-mx}\sum_{k=0}^{\infty}(k-mx)\frac{(mx)^{k}}{k!}\left(1-\frac{\Gamma(1+k,mx)}{\Gamma(1+k)}\right)\nonumber\\
        &=-m^{-1}e^{-mx}\sum_{k=0}^{\infty}(k-mx)\frac{(mx)^{k}}{k!}\frac{\Gamma(1+k,mx)}{\Gamma(1+k)}.
    \end{align*}

    Using $\frac{\Gamma(1+k,mx)}{\Gamma(1+k)}=\P(Y \leq k)\in [0,1]$ for $Y \sim \Poi(mx)$, the above calculation yields
    \begin{align*}
        |G_m^S(x)|&\leq m^{-1/2}e^{-(a+m)x}\sum_{k=0}^{\infty}|k-mx|\frac{(mx)^{k}}{k!}\frac{\Gamma(1+k,mx)}{\Gamma(1+k)}\nonumber\\
        &\leq m^{-1/2}e^{-ax}\sum_{k=0}^{\infty}|k-mx|V_{k,m}(x)\nonumber\\
        &\leq m^{-1/2}e^{-ax}\left(\sum_{k=0}^{\infty}(k-mx)^2V_{k,m}(x)\right)^{1/2}\nonumber\\
        &=m^{-1/2}e^{-ax}\sqrt{mx}=\sqrt{x}e^{-ax}.
    \end{align*}
    This is integrable since
    \begin{equation*}
        \int_0^{\infty}\sqrt{x}e^{-ax}\diff x=\frac{\sqrt{\pi}}{2a^{3/2}}.
    \end{equation*}
    With the dominated convergence theorem it follows that
    \begin{equation*}
        \int_0^{\infty}|G_m^S(x)-G^S(x)|\diff x=o(1)
    \end{equation*}
    and
    \begin{align*}
        &\left|\int_0^{\infty}g(x)G_m^S(x)\diff x-\int_0^{\infty}g(x)G^S(x)\diff x\right|\nonumber\\
        &\leq \sup_{x\in[0,\infty)}|g(x)|\int_0^{\infty}|G_m^S(x)-G^S(x)|\diff x=o(1),
    \end{align*}
    as $g$ is bounded.

    The proof for $\tilde{R}_{1,m}^S$ is very similar with Eq.(\ref{Eq:Decomposition}) and the fact that the second term is negligible.
  \end{enumerate}
\end{proof}

\subsection*{Proofs}
\label{Section:ProofsSzasz}

\begin{proof}[Proof of \autoref{Theorem:BiasVarS}]
  The bias follows directly from \autoref{Lemma:LorentzS}. \noindent For the proof of the variance, some ideas are taken from \cite{ouimetLocalLimitTheorem2020} and \cite{leblancEstimatingDistributionFunctions2012}. Let
  \begin{equation*}
    Y_{i,m}^S=\sum_{k=0}^{\infty}\Delta_i\left(\frac{k}{m}\right)V_{k,m}(x),
  \end{equation*}
  where $\Delta_i(x)=\mathbb{I}(X_i \leq x)-F(x)$ for $x\in[0,\infty)$. We know that $\Delta_1,...,\Delta_n$ are i.i.d. with mean zero.
  Hence,
  \begin{align}
    \label{Eq:DarstellungSummeS}
    \hat{F}_{m,n}^S(x)-S_m(F;x) &=\sum_{k=0}^{\infty}\left[F_n\left(\frac{k}{m}\right)-F\left(\frac{k}{m}\right)\right]V_{k,m}(x)\nonumber\\
    &=\frac{1}{n}\sum_{k=0}^{\infty}\sum_{i=1}^n\left[\mathbb{I}\left(X_i\leq\frac{k}{m}\right)-F\left(\frac{k}{m}\right)\right]V_{k,m}(x)\nonumber\\
    &=\frac{1}{n}\sum_{i=1}^nY_{i,m}^S.
  \end{align}
  Note that $Y_{i,m}^S<\infty$ a.s. and that, for given $m$, $Y_{1,m}^S,...,Y_{n,m}^S$ are i.i.d. with mean zero.
  This means that the variance can be calculated by
  \begin{align}
    \label{Eq:VarYS}
    \Var\left[\hat{F}_{m,n}^S(x)\right]&=\Var\left[\hat{F}_{m,n}^S(x)-S_m(F;x)\right]=\frac{1}{n^2}\sum_{i=1}^n\Var[Y_{i,m}^S]\nonumber\\
    &=\frac{1}{n}\Var\left[Y_{1,m}^S\right]=\frac{1}{n}\E\left[\left(Y_{1,m}^S\right)^2\right].
  \end{align}
It also holds for $x,y \in [0,\infty)$ that
\begin{align*}
  \E[\Delta_1(x)\Delta_1(y)]&=\E[(\mathbb{I}(X_1 \leq x)-F(x))(\mathbb{I}(X_1 \leq y)-F(y))]\nonumber\\
  &=\min(F(x),F(y))-F(x)F(y),
\end{align*}
which implies
\begin{align}
  \label{Eq:EWYS}
  \E\left[\left(Y_{1,m}^S\right)^2\right]&=\sum_{k=0}^\infty\sum_{l=0}^\infty\E\left[\Delta_1\left(\frac{k}{m}\right)\Delta_1\left(\frac{l}{m}\right)\right]V_{k,m}(x)V_{l,m}(x)\nonumber\\
  &=\sum_{k=0}^m\sum_{l=0}^\infty\min\left(F\left(\frac{k}{m}\right), F\left(\frac{l}{m}\right)\right)V_{k,m}(x)V_{l,m}(x)\nonumber\\
  &\enspace-\sum_{k=0}^\infty\sum_{l=0}^\infty F\left(\frac{k}{m}\right)F\left(\frac{l}{m}\right)V_{k,m}(x)V_{l,m}(x)\nonumber\\
  &=\sum_{k=0}^\infty F\left(\frac{k}{m}\right)V_{k,m}^2(x)\nonumber\\
  &\enspace+2 \mathop{\sum\sum}_{0\leq k < l \leq \infty} F\left(\frac{k}{m}\right)V_{k,m}(x)V_{l,m}(x)-S_m^2(x)\nonumber\\
  &=\sum_{k,l=0}^{\infty}F\left(\frac{k \wedge l}{m}\right)V_{k,m}(x)V_{l,m}(x)-S_m^2(x)).
\end{align}
Use Taylor's theorem to get
\begin{equation*}
  F\left(\frac{k \wedge l}{m}\right)=F(x)+\left(\frac{k \wedge l}{m}-x\right)f(x)+O\left(\left(\frac{k \wedge l}{m}-x\right)^2\right).
\end{equation*}
With this, we get
\begin{align}
  \label{Eq:NeueDarstellung}
  \E\left[\left(Y_{1,m}^S\right)^2\right]&=F(x)+f(x)\sum_{k,l=0}^{\infty}\left(\frac{k \wedge l}{m}-x\right)V_{k,m}(x)V_{l,m}(x)\nonumber\\
  &\enspace +O\left(\sum_{k,l=0}^{\infty}\left|\frac{k}{m}-x\right|\left|\frac{l}{m}-x\right|V_{k,m}(x)V_{l,m}(x))\right)-S_m^2(x)\nonumber\\
  &=\sigma^2(x)+m^{-1/2}f(x)\tilde{R}_{1,m}^S+O_x(m^{-1})\nonumber\\
  &=\sigma^2(x)-f(x)m^{-1/2}\sqrt{\frac{x}{\pi}}+o_x(m^{-1/2}).
\end{align}
We used the fact that with the Cauchy-Schwarz-Inequality and \autoref{Eq:QuadratS} we get
\begin{align*}
  \sum_{k,l=0}^{\infty}\left|\frac{k}{m}-x\right|\left|\frac{l}{m}-x\right|V_{k,m}(x)V_{l,m}(x)) \leq m^{-2}\sum_{k=0}^{\infty}(k-mx)^2V_{k,m}(x)=\frac{x}{m}.
\end{align*}
This proves the claim.
\end{proof}

\begin{proof}[Proof of \autoref{Theorem:AsySzasz}]
    This proof follows the proof of Theorem 2 in \cite{leblancEstimatingDistributionFunctions2012}.
    For fixed $m$ we know from the proof of \autoref{Theorem:BiasVarS} that
    \begin{equation*}
        \hat{F}_{m,n}^S(x)-S_m(F;x) =\frac{1}{n}\sum_{i=1}^nY_{i,m}^S,
    \end{equation*}
    where the $Y_{i,m}^S$ are i.i.d. random variables with mean zero. Define $(\gamma_m^S)^2=\E[(Y_{1,m}^S)^2]$.
    We now use the central limit theorem for double arrays (see \cite{serflingApproximationTheoremsMathematical1980}, Section 1.9.3) to show the claim.

    Defining
    \begin{equation*}
        A_n=\E\left[\sum_{i=1}^nY_{i,m}^S\right]=0 \enspace \text{and} \enspace B_n^2=\Var\left[\sum_{i=1}^nY_{i,m}^S\right]=n \left(\gamma_m^S\right)^2,
    \end{equation*}
    it says that
    \begin{equation*}
        \frac{\sum_{i=1}^nY_{i,m}^S-A_n}{B_n} \xrightarrow{D} \mathcal{N}(0,1)
    \end{equation*}
    if and only if the Lindeberg condition
    \begin{equation*}
        \frac{n\E\left[\mathbb{I}(|Y_{1,m}^S|>\epsilon B_n)\left(Y_{1,m}^S\right)^2\right]}{B_n^2} \rightarrow 0 \enspace \text{for} \enspace n \rightarrow \infty \enspace \text{and all} \enspace \epsilon> 0
    \end{equation*}
    is satisfied.
    With \autoref{Eq:NeueDarstellung} we know that $\gamma_m^S \rightarrow \sigma(x)$ for $m \rightarrow \infty$ (which follows from $n \rightarrow \infty$) and it holds for $n \rightarrow \infty$ that
    \begin{align*}
        \frac{\sum_{i=1}^nY_{i,m}^S-A_n}{B_n} &\xrightarrow{D} \mathcal{N}(0,1)\nonumber\\
        \Leftrightarrow \frac{\sum_{i=1}^nY_{i,m}^S}{\sqrt{n}\cdot\gamma_m^S} &\xrightarrow{D} \mathcal{N}(0,1)\nonumber\\
        \Leftrightarrow \frac{\sqrt{n}}{\gamma_m^S} \left(\hat{F}_{m,n}^S(x)-S_m(F;x)\right) &\xrightarrow{D} \mathcal{N}(0,1) \nonumber \\
        \Leftrightarrow \sqrt{n}\left(\hat{F}_{m,n}^S(x)-S_m(F;x)\right) &\xrightarrow{D} \mathcal{N}\left(0,\sigma^2(x)\right),
    \end{align*}
    which is the claim of \autoref{Theorem:AsySzasz}.
    In our case the Lindeberg condition has the form
    \begin{equation*}
        \frac{\E[\mathbb{I}(|Y_{1,m}^S|>\epsilon \sqrt{n}\gamma_m^S)\left(Y_{1,m}^S\right)^2]}{\left(\gamma_m^S\right)^2} \rightarrow 0 \enspace \text{for} \enspace n \rightarrow \infty \enspace \text{and all} \enspace \epsilon> 0.
    \end{equation*}
    This is what has to be shown to proof the theorem. Using the fact that
    \begin{equation*}
        |Y_{1,m}^S| \leq \sum_{k=0}^\infty \left|\Delta_1\left(\frac{k}{m}\right)\right| V_{k,m}(x) \leq \sum_{k=0}^\infty V_{k,m}(x) = 1
    \end{equation*}
    leads to
    \begin{equation*}
        \mathbb{I}\left(|Y_{1,m}^S|>\epsilon \sqrt{n}\gamma_m^S\right) \leq \mathbb{I}\left(1
        >\epsilon \sqrt{n}\gamma_m^S\right) \rightarrow 0,
    \end{equation*}
    which gives the desired result.
\end{proof}

\begin{proof}[Proof of  \autoref{Theorem:MISES}]
    This proof follows the proof of Theorem 3 in \cite{leblancEstimatingDistributionFunctions2012}.
    We now use the asymptotic expression for $\tilde{R}_{1,m}^S$. Using \autoref{Eq:NeueDarstellung} and\autoref{Lemma:LorentzS} leads to
    \begin{align*}
        &\MISE\left[\hat{F}^S_{m,n}\right]\nonumber\\
        &=\int_0^{\infty} \left[\Var\left[\hat{F}_{m,n}^S(x)\right]+\Bias\left[\hat{F}_{m,n}^S(x)\right]^2\right]e^{-ax}f(x)\diff x \nonumber\\
        &=\frac{1}{n}\int_0^{\infty}\left[\sigma^2(x)+m^{-1/2}f(x)\tilde{R}_{1,m}^S(x)+O_x(m^{-1})\right]e^{-ax}f(x)\diff x\nonumber\\
        &\;+\int_0^{\infty}\left[m^{-1}b^S(x)+o\left(\frac{x}{m}\right)\right]^2e^{-ax}f(x)\diff x\nonumber\\
        &=\frac{1}{n}\int_0^{\infty}\left[\sigma^2(x)+m^{-1/2}f(x)\tilde{R}_{1,m}^S(x)\right]e^{-ax}f(x)\diff x+\int_0^{\infty}m^{-2}(b^S(x))^2e^{-ax}f(x)\diff x\nonumber\\
        &\;+O(m^{-1}n^{-1})+o(m^{-2}).
    \end{align*}
    Now, with $f(x)\frac{\sqrt{x}}{\sqrt{\pi}}=V^S(x)$ and \autoref{Lemma:EigenschaftenLRS} (h), we get
    \begin{align*}
        \MISE\left[\hat{F}_{m,n}^S\right]=n^{-1}C_1^S-n^{-1}m^{-1/2}C_2^S+m^{-2}C_3^S+o(m^{-2})+o(m^{-1/2}n^{-1}).
    \end{align*}
    The integrals $C_i^S$ exist for $i=1,2,3$ because $f$ and $(f')^2$ are  positive and bounded on $[0,\infty)$. It follows that
    \begin{equation*}
        C_1^S=\int_0^{\infty}F(x)(1-F(x))e^{-ax}f(x)\diff x\leq \|f\|\int_0^{\infty}e^{-ax}\diff x=\frac{\|f\|}{a}<\infty,
    \end{equation*}
    \begin{equation*}
      C_2^S=\int_0^{\infty}f(x)\left[\frac{x}{\pi}\right]^{1/2}e^{-ax}f(x)\diff x\leq \frac{\|f\|^2}{\sqrt{\pi}}\int_0^{\infty}\sqrt{x}e^{-ax}\diff x=\frac{\|f\|^2}{2a^{3/2}}<\infty,
    \end{equation*}
    and
    \begin{align*}
       C_3^S&=\int_0^{\infty} \left(\frac{xf'(x)}{2}\right)^2e^{-ax}f(x)\diff x\nonumber\\
       &\leq \frac{\|(f')^2\|\cdot\|f\|}{4}\int_0^{\infty} x^2e^{-ax}\diff x=\frac{\|f'\|^2\|f\|}{2a^3}<\infty,
    \end{align*}
    where the norm is again defined by $\|g\|=\displaystyle\sup_{x \in [0,\infty)}|g(x)|$ for a bounded function $g:[0,\infty)\rightarrow \R$.
\end{proof}

\begin{proof}[Proof of  \autoref{Theorem:DeficiencyS}]
    This proof follows the proof of Theorem 4 in \cite{leblancEstimatingDistributionFunctions2012}. We only present the proof for the local part.
    For simplicity, write $i(n)=i_L^S(n,x)$.

    By the definition of $i(n)$ we know that $\displaystyle\lim_{n \rightarrow \infty}i(n)=\infty$ and
    \begin{align}
        \label{Eq:iLS}
        &\MSE\left[F_{i(n)}(x)\right] \leq \MSE\left[\hat{F}_{m,n}^S(x)\right] \leq \MSE\left[F_{i(n)-1}(x)\right]\nonumber\\
        \Leftrightarrow & \enspace i(n)^{-1}\sigma^2(x) \leq n^{-1}\sigma^2(x)-m^{-1/2}n^{-1}V^S(x)+m^{-2}(b^S(x))^2\nonumber\\
        &+o_x(m^{-1/2}n^{-1})+o_x(m^{-2})\leq (i(n)-1)^{-1}\sigma^2(x)\nonumber\\
        \Leftrightarrow & \enspace 1 \leq \frac{i(n)}{n}\left[1-m^{-1/2}\theta^S(x)+m^{-2}n\gamma^S(x)+o_x(m^{-1/2})+o_x(m^{-2}n)\right]\nonumber\\
        &\enspace\leq\frac{i(n)}{i(n)-1},
    \end{align}
    where $\theta^S(x)=\frac{V^S(x)}{\sigma^2(x)}$ and $\gamma^S(x)=\frac{(b^S(x))^2}{\sigma^2(x)}$.
    Now, if $mn^{-1/2} \rightarrow \infty$ ($\Leftrightarrow m^{-2}n \rightarrow 0$), taking the limit $n\rightarrow \infty$ leads to
    \begin{equation*}
        \label{Eq:LimitS}
        \frac{i(n)}{n} \rightarrow 1,
    \end{equation*}
    so that
    \begin{equation*}
        i(n)=n+o_x(n)=n(1+o_x(1)).
    \end{equation*}

    \begin{enumerate}
    \item[(a)]
    We assume that $mn^{-2/3}\rightarrow \infty$ and $mn^{-2}\rightarrow 0$.
    Rewrite \autoref{Eq:iLS} as
    \begin{align}
        \label{Eq:InequalityS}
        &m^{-1/2}n^{-1}\theta^S(x)\leq A_{1,n}+m^{-2}\gamma^S(x)+o_x(m^{-1/2}n^{-1})+o_x(m^{-2})\nonumber\\
        &\leq m^{-1/2}n^{-1}\theta^S(x)+A_{2,n}\nonumber\\
        \Leftrightarrow & \enspace\theta^S(x) \leq m^{1/2}nA_{1,n}+m^{-3/2}n\gamma^S(x)+o_x(1)+o_x(m^{-3/2}n)\nonumber\\
        &\leq \theta^S(x)+m^{1/2}nA_{2,n},
    \end{align}
    where
    \begin{equation*}
        A_{1,n}=\frac{1}{n}-\frac{1}{i(n)} \enspace \text{and} \enspace A_{2,n}=\frac{1}{i(n)-1}-\frac{1}{i(n)}.
    \end{equation*}
    It holds that
    \begin{equation*}
        \label{Eq:A1nS}
        \lim_{n\rightarrow \infty} m^{1/2}nA_{1,n}=\left(\lim_{n\rightarrow \infty}\frac{i(n)-n}{m^{-1/2}n}\right)\left(\lim_{n\rightarrow \infty}\frac{n}{i(n)}\right)=\lim_{n\rightarrow \infty}\frac{i(n)-n}{m^{-1/2}n},
    \end{equation*}
    and because $m^{1/2}n^{-1}=(mn^{-2})^{1/2} \rightarrow 0$,
    \begin{equation*}
        \label{Eq:A2nS}
        \lim_{n\rightarrow \infty} m^{1/2}nA_{2,n}=\left(\lim_{n\rightarrow \infty}m^{1/2}n^{-1}\right)\left(\lim_{n\rightarrow \infty}\frac{n}{i(n)}\right)\left(\lim_{n\rightarrow \infty}\frac{n}{i(n)-1}\right)=0.
    \end{equation*}
    We also know that $m^{-3/2}n=(mn^{-2/3})^{-3/2}\rightarrow 0$, hence
    \begin{equation*}
        \lim_{n\rightarrow \infty}\frac{i(n)-n}{m^{-1/2}n}=\theta^S(x) \Rightarrow \frac{i(n)-n}{m^{-1/2}n}=\theta^S(x)+o_x(1)
    \end{equation*}
    follows from \Eq{Eq:InequalityS}.

    \item[(b)]
    The second part can be proven with similar arguments. If $mn^{-2/3} \rightarrow c$ it also holds that $m^{-2}n=(mn^{-2/3})^{-3/2}m^{-1/2} \rightarrow 0$ and $m^{1/2}n^{-1}=(mn^{-2/3})^{1/2}n^{-2/3} \rightarrow 0$ so that we get that
    \begin{equation*}
        \lim_{n\rightarrow \infty}\frac{i(n)-n}{m^{-1/2}n}=\theta^S(x)-c^{-3/2}\gamma^S(x)
    \end{equation*}
    and with
    \begin{align*}
        \lim_{n\rightarrow \infty}\frac{i(n)-n}{m^{-1/2}n}&=\left(\lim_{n\rightarrow \infty}\frac{i(n)-n}{n^{2/3}}\right)\left(\lim_{n\rightarrow \infty}m^{1/2}n^{-1/3}\right)\nonumber\\
        &=c^{1/2}\lim_{n\rightarrow \infty}\frac{i(n)-n}{n^{2/3}}
    \end{align*}
    the claim
    \begin{equation*}
        c^{1/2}\frac{i(n)-n}{n^{2/3}}=\theta^S(x)-c^{-3/2}\gamma^S(x)+o_x(1)
    \end{equation*}
    holds.
    \end{enumerate}

    The global part can be proved analogously with $\tilde{\theta}^S=\frac{C_2^S}{C_1^S}$ and $\tilde{\gamma}^S=\frac{C_3^S}{C_1^S}$.
\end{proof}

\subsection*{Proof Asymptotic Normality of the Hermite Estimator on the Half Line}

This proof takes some ideas from the proof of Theorem 2 in \cite{leblancEstimatingDistributionFunctions2012}.
    For fixed $N$ it holds that
    \begin{align*}
        \hat{F}_{N,n}^H(x)-\E\left[\hat{F}_{N,n}^H(x)\right]
        &=\int_0^x\sum_{k=0}^N\hat{a}_kh_k(t)\diff t-\int_0^x\sum_{k=0}^Na_kh_k(t)\diff t\\
        &=\int_0^x\sum_{k=0}^N\left[\frac{1}{n}\sum_{i=1}^nh_k(X_i)\right]h_k(t)\diff t-\int_0^x\sum_{k=0}^Na_kh_k(t)\diff t\\
        &=\frac{1}{n}\sum_{i=1}^n\left[\int_0^xT_N(X_i,t)\diff t -\int_0^x\sum_{k=0}^Na_kh_k(t)\diff t\right]\\
        &=\frac{1}{n}\sum_{i=1}^nY_{i,N},
    \end{align*}
    where
    \begin{equation*}
        T_N(x,y)=\sum_{k=0}^Nh_k(x)h_k(y)
    \end{equation*}
    and
    \begin{equation*}
        Y_{i,N}=\int_0^x\left[T_N(X_i,t)-\sum_{k=0}^Na_kh_k(t)\right]\diff t, i \in \{1,...,n\}.
    \end{equation*}
    The $Y_{i,N}$ are i.i.d. random variables with mean $0$. Define $\gamma_N^2=\E[Y_{1,N}^2]$.
    We use the central limit theorem for double arrays (see \cite{serflingApproximationTheoremsMathematical1980}, Section 1.9.3) to show the claim.
    Defining
    \begin{equation*}
        A_n=\E\left[\sum_{i=1}^nY_{i,N}\right]=0 \enspace \text{and} \enspace B_n^2=\Var\left[\sum_{i=1}^nY_{i,N}\right]=n \gamma_N^2,
    \end{equation*}
    it says that
    \begin{equation*}
        \frac{\sum_{i=1}^nY_{i,N}-A_n}{B_n} \xrightarrow{D} \mathcal{N}(0,1)
    \end{equation*}
    if and only if the Lindeberg condition
    \begin{equation*}
        \frac{n\E[\mathbb{I}(|Y_{1,N}|>\epsilon B_n)Y_{1,N}^2]}{B_n^2} \rightarrow 0 \enspace \text{for} \enspace n \rightarrow \infty \enspace \text{and all} \enspace \epsilon> 0
    \end{equation*}
    is satisfied.
    It holds for $n \rightarrow \infty$ that
    \begin{align*}
        \frac{\sum_{i=1}^nY_{i,N}-A_n}{B_n} \xrightarrow{D} \mathcal{N}(0,1)
        & \Leftrightarrow \frac{\sum_{i=1}^nY_{i,N}}{\sqrt{n}\cdot \gamma_N} \xrightarrow{D} \mathcal{N}(0,1) \\
        &\Leftrightarrow \frac{\sqrt{n}}{\gamma_N} \left(\hat{F}_{N,n}^H(x)-\E\left[\hat{F}_{N,n}^H(x)\right]\right) \xrightarrow{D} \mathcal{N}(0,1) \\
        &\Leftrightarrow \sqrt{n}\left(\hat{F}_{N,n}^H(x)-\E\left[\hat{F}_{N,n}^H(x)\right]\right) \xrightarrow{D} \mathcal{N}\left(0,\sigma^2(x)\right).
    \end{align*}
    The last equivalence holds because of the following. We have to calculate $\gamma_N^2$ which is given by
    \begin{equation}
        \label{Eq:GammaHalf}
        \begin{gathered}
            \gamma_N^2=\E\left[\left(\int_0^xT_N(X_1,t)\diff t-\int_0^x\sum_{k=0}^Na_kh_k(t)\diff t\right)^2\right]\\
            =\E\left[\left(\int_0^xT_N(X_1,t)\diff t\right)^2\right]-2\int_0^x\sum_{k=0}^Na_kh_k(t)\diff t\cdot\E\left[\int_0^xT_N(X_1,t)\diff t\right]\\
            +\left(\int_0^x\sum_{k=0}^Na_kh_k(t)\diff t\right)^2.
        \end{gathered}
    \end{equation}
    The first part is the only part where we do not know the asymptotic behavior. Hence, we now take a closer look at this part.
    With Eq.(A8) in \cite{liebscherHermiteSeriesEstimators1990}, which only holds on compact sets, we know that
   \begin{align}
        \label{Eq:GammaFirstPart}
            \E\left[\left(\int_0^xT_N(X_1,t)\diff t\right)^2\right]
            =& \lim_{P\rightarrow \infty} \int_0^P\left[\int\limits_0^x\frac{\sin (M(r-t))}{\pi(r-t)}+O(N^{-1/2})\diff t\right]^2f(r)\diff r \nonumber\\
            =& \int_0^\infty\left[\int\limits_0^x\frac{\sin (M(r-t))}{\pi(r-t)}+O(N^{-1/2})\diff t\right]^2f(r)\diff r \nonumber\\
            =& \int_0^x\left[\int\limits_0^x\frac{\sin (M(r-t))}{\pi(r-t)}+O(N^{-1/2})\diff t\right]^2f(r)\diff r \nonumber \\
             &+\int_x^\infty\left[\int\limits_0^x\frac{\sin (M(r-t))}{\pi(r-t)}+O(N^{-1/2})\diff t\right]^2f(r)\diff r,
    \end{align}
    where $M=\frac{\sqrt{2n+3}+\sqrt{2n+1}}{2}$.
    The inner integral can be written as
    \begin{equation*}
        \int_0^x\frac{\sin (M(r-t))}{\pi(r-t)}\diff t=\int_{M(r-x)}^{Mr}\frac{\sin(l)}{\pi l}\diff l
    \end{equation*}
    and with the fact that for $M \rightarrow \infty$,
    \begin{equation*}
        \int\limits_{Ma}^{Mb}\frac{\sin(l)}{\pi l}dl \rightarrow
        \begin{cases}
            1, & a < 0 < b,\\
            0, & 0 < a < b,\\
            0, & a < b < 0,
        \end{cases}
    \end{equation*}
    it follows with \autoref{Eq:GammaFirstPart} for $n \rightarrow \infty$ (which implies $M \rightarrow \infty$) that
    \begin{equation}
        \label{Eq:Limit}
        \E\left[\left(\int_0^xT_N(X_1,t)\diff t\right)^2\right] \rightarrow \int_0^x f(r)\diff r = F(x).
    \end{equation}
    In the end of the proof, it is explained in detail why it is possible to move the limit $M \to \infty$ inside the integral. Then, plugging \autoref{Eq:Limit} in \autoref{Eq:GammaHalf} and using the fact that we know limits of the other parts from Lemma 1 in \cite{greblickiPointwiseConsistencyHermite1985}, it holds for $n \rightarrow \infty$ that
    \begin{equation*}
        \label{Eq:GammaErgebnisHalf}
        \gamma_N^2 \rightarrow F(x)-2F(x)^2+F(x)^2=\sigma^2(x).
    \end{equation*}
    Now, we have to show that asymptotic normality actually holds.
    In our case the Lindeberg condition has the form
    \begin{equation*}
        \frac{\E\left[\mathbb{I}(|Y_{1,N}|>\epsilon \sqrt{n}\gamma_N)Y_{1,N}^2\right]}{\gamma_N^2} \rightarrow 0 \enspace \text{for} \enspace n \rightarrow \infty \enspace \text{and all} \enspace \epsilon> 0.
    \end{equation*}
    This is what has to be shown to prove the theorem. Writing the expected value as an integral, we get
    \begin{align*}
        \int_0^{\infty} \mathbb{I}\left(\left|A_N(r)\right|>\epsilon \sqrt{n}\gamma_N\right)A_N(r)^2f(r)\diff r
    \end{align*}
    with
    \begin{equation*}
        A_{N}(r)=\int_0^x\left[T_N(r,t)-\sum_{k=0}^Na_kh_k(t)\right]\diff t.
    \end{equation*}
    With the arguments from above, the left side of the inequality in the indicator function is bounded by a constant, depending on $x$, for large $n$.
    Using this result, we get for large $n$ that
    \begin{align*}
        \frac{\E\left[\mathbb{I}(|Y_{1,N}|>\epsilon \sqrt{n}\gamma_N)Y_{1,N}^2\right]}{\gamma_N^2} &\leq \mathbb{I}(c_x>\epsilon \sqrt{n} \gamma_N)\frac{\E\left[Y_{1,N}^2\right]}{\gamma_N^2}\nonumber\\
        &=\mathbb{I}\left(\frac{c_x}{\sqrt{n}\gamma_N}>\epsilon\right) \rightarrow 0,
    \end{align*}
    where $c_x$ is a constant depending on $x$, which proves the claim.

    We explain now, why it is possible to exchange limit and integral in the calculation of the limit of $\gamma_N$. We first observe that for $x \neq 0$,
\begin{equation*}
    -\frac{1}{\pi |x|}-\frac{1}{2} \leq \int_0^x\frac{\sin(l)}{\pi l}\diff l \leq \frac{1}{\pi |x|}+\frac{1}{2}.
\end{equation*}
It follows that
\begin{equation*}
    \left(\int_0^x\frac{\sin(l)}{\pi l}\diff l\right)^2 \leq \left(\frac{1}{\pi |x|}+\frac{1}{2}\right)^2.
\end{equation*}
Hence, for $r \in \{0,x\}$,
\begin{equation*}
    \left(\int_{M(r-x)}^{Mr}\frac{\sin(l)}{\pi l}\diff l\right)^2 \leq \left(\frac{1}{\pi |M x|}+\frac{1}{2}\right)^2,
\end{equation*}
and for the rest,
\begin{align*}
    &\left(\int_{M(r-x)}^{Mr}\frac{\sin(l)}{\pi l}\diff l\right)^2 = \left(\int_0^{M r}\frac{\sin(l)}{\pi l}\diff l-\int_0^{M(r-x)}\frac{\sin(l)}{\pi l}\diff l\right)^2\nonumber\\
    &\leq \left(\frac{1}{\pi |M r|}+\frac{1}{2}\right)^2 + 2\left(\frac{1}{\pi |M r|}+\frac{1}{2}\right)\left(\frac{1}{\pi |M (r-x)|}+\frac{1}{2}\right)\nonumber\\
    &\enspace+\left(\frac{1}{\pi |M (r-x)|}+\frac{1}{2}\right)^2.
\end{align*}

\SingleFigureHere{0.75\textwidth}{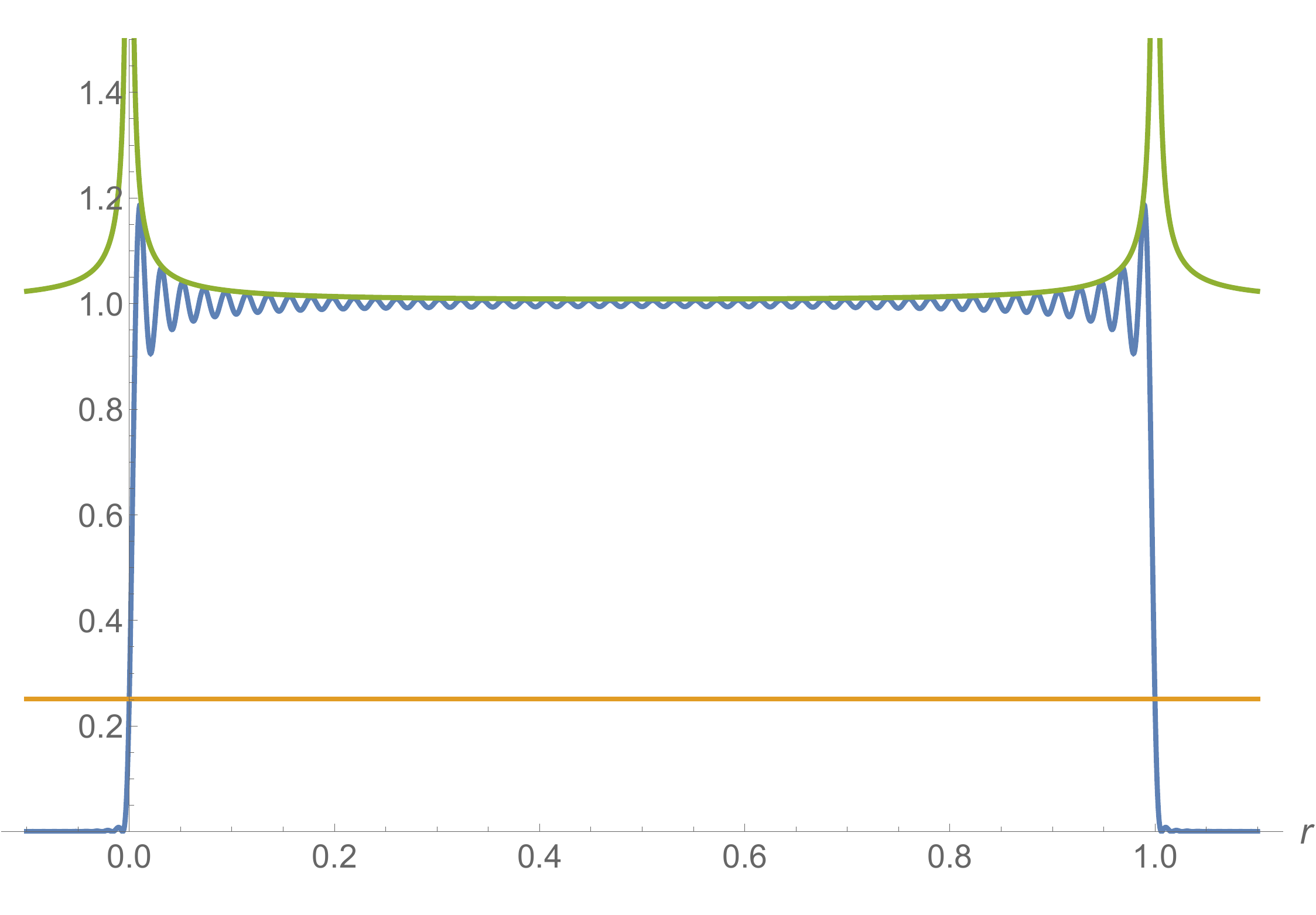}{Illustration of the bounds for $M=300$, $x=1$.}{MathematicaHermite}

In \autoref{MathematicaHermite}, the two bounds calculated above are illustrated. The orange line is the bound for $r \in \{0,x\}$ and the green line is the bound for the rest. The only critical parts are close to $r=0$ and $r=x$, where the function attains its maximum.
It is obvious that the maximum value is given by
\begin{align*}
    \left(\frac{1}{\pi |M r_{\max}|}+\frac{1}{2}\right)^2 &+ 2\left(\frac{1}{\pi |M r_{\max}|}+\frac{1}{2}\right)\left(\frac{1}{\pi |M (r_{\max}-x)|}+\frac{1}{2}\right)\nonumber\\
    &+\left(\frac{1}{\pi |M (r_{\max}-x)|}+\frac{1}{2}\right)^2,
\end{align*}
where the function attains the maximum value in $r_{\max}$. Now, for $M \geq M_0$, this is bounded by
\begin{align*}
    \left(\frac{1}{\pi |M_0 r_{\max}|}+\frac{1}{2}\right)^2 &+ 2\left(\frac{1}{\pi |M_0 r_{\max}|}+\frac{1}{2}\right)\left(\frac{1}{\pi |M_0 (r_{\max}-x)|}+\frac{1}{2}\right)\nonumber\\
    &+\left(\frac{1}{\pi |M_0 (r_{\max}-x)|}+\frac{1}{2}\right)^2.
\end{align*}
The part $O(N^{-\frac{1}{2}})$ in \autoref{Eq:GammaFirstPart} is very small for large $M \geq M_0$ and does not change the fact that the function is bounded. We call the bound $d_x$.
This is a function that is integrable because
\begin{equation*}
    \int_0^{\infty}d_xf(r)dr=d_x < \infty.
\end{equation*}
With the dominated convergence theorem, it is possible to move the limit over $M$ inside the integral.

\bibliographystyle{spbasic}
\bibliography{Referenzen}

\end{document}